\documentclass[12pt,twoside]{article}
\usepackage{amsmath,amsfonts,amssymb,a4,enumitem,epsfig,array,psfrag}
\usepackage{url}
\usepackage{amsthm}
\usepackage{subfig}
\usepackage{chngcntr} 
\usepackage{color}
\usepackage[section]{placeins}

\usepackage[T1]{fontenc}          
\usepackage{float}

\DeclareMathAlphabet{\mathpzc}{OT1}{pzc}{m}{it}

\newcommand{\R}{\mathbb{R}}

\newtheorem{teo}{Theorem}[section]
\newtheorem{Main}{Theorem}

\newtheorem{proposition}[teo]{Proposition}

\newtheorem{example}[teo]{Example}
\newtheorem{Rem}[teo]{Remark}

\newtheorem{definition}[teo]{Definition}

\newtheorem{lemma}[teo]{Lemma}

\def\N{\mathbb N}

\def\SS{\mathbb S}
\def\H{\mathbb H}
\def\1{\mathbf 1}
\def\k{\mathbf k}
\def\j{\mathbf j}
\def\i{\mathbf i}

\begin{document}
\title{Results on the homotopy type of the spaces of locally convex curves on $\SS^3$}
\author{
Emília Alves \and Nicolau C. Saldanha }

\maketitle

\begin{abstract}

A curve $\gamma: [0,1] \rightarrow \SS^n$ of class $C^k$ ($k \geqslant n$)
is locally convex if the vectors
$\gamma(t), \gamma'(t), \gamma''(t), \cdots, \gamma^{(n)}(t)$
are a positive
basis to $\mathbb{R}^{n+1}$ for all $t \in [0,1]$.
Given an integer $n \geq 2$ and $Q \in \mathrm{SO}_{n+1}$,
let $\mathcal{L}\SS^n(Q)$ be
the set of all locally convex curves $\gamma: [0,1] \rightarrow \SS^n$
with fixed initial and final Frenet frame
$\mathcal{F}_\gamma(0)=I$ and $\mathcal{F}_\gamma(1)=Q$.
Saldanha and Shapiro proved that there are just finitely many non-homeomorphic spaces among $\mathcal{L}\mathbb{S}^n(Q)$ when $Q$ varies in $\mathrm{SO}_{n+1}$ (in particular, at most $3$ for $n=3$).
For any $n \geqslant 2$, one of these spaces is proved to be homeomorphic
to the (well understood) space of generic curves (see below),
but very little is known in general about the others.
For $n=2$, Saldanha determined the homotopy type of the spaces $\mathcal{L}\mathbb{S}^2(Q)$. The purpose of this work is to study the case $n=3$.
We will obtain information on the homotopy type
of one of these two other spaces,
allowing us to conclude that
none of the connected components of $\mathcal{L}\mathbb{S}^3(-I)$
is homeomorphic to a connected component of the space of generic curves.

\end{abstract}

\section{Introduction}

A curve $\gamma:[0,1] \rightarrow \SS^3$ of class $C^k$ ($k \geq 3$) is called \emph{locally convex} if its geodesic torsion is always positive, or equivalently, if $ \mathrm{det}(\gamma(t),\gamma'(t),\gamma''(t), \gamma'''(t)) > 0 $ 
for all $t$. For $Q \in \ \mathrm{SO}_4$, let $\mathcal{L}\SS^{3}(Q)$ be the set of all locally convex curves $\gamma$ with $\gamma(0) = e_1$, $\gamma(1) = Qe_1$, $\gamma'(0) = e_2$, $\gamma'(1) = Qe_2$ and $\gamma''(0) = e_3$, $\gamma''(1) = Qe_3$.  
Shapiro and Anisov proved that $\mathcal{L}\SS^3(-I)$
(where $I$ is the identity matrix)
has three connected components, that we denote by $\mathcal{L}\SS^3(\1,-\1)_c$, $\mathcal{L}\SS^3(-\1,\1)$ and $\mathcal{L}\SS^3(\1,-\1)_n$, where $\mathcal{L}\SS^3(\1,-\1)_c$ is the set of convex curves, which is contractible
(this notation will be clarified later).
Our aim is to understand the two other spaces.
Even though we do not have a complete answer yet,
in this work we present new partial results.

The space $\mathcal{L}\SS^3(I)$ has two connected components:
$\mathcal{L}\SS^3(\1,\1)$ and $\mathcal{L}\SS^3(-\1,-\1)$.
For \[Q_0 =
\begin{pmatrix}
-1 & 0 & 0 & 0\\
0 & -1 & 0 & 0 \\
0 & 0 & 1 & 0 \\
0 & 0 & 0 & 1
\end{pmatrix},\]
the space $\mathcal{L}\SS^3(Q_0)$ has two connected components $\mathcal{L}\SS^3(\mathbf{i},-\mathbf{i})$ and $\mathcal{L}\SS^3(-\mathbf{i},\mathbf{i})$.
It follows from \cite{SS12} that for, any $Q \in \mathrm{SO}_4$,
the space $\mathcal{L}\SS^3(Q)$ is homeomorphic to one of these: 
$\mathcal{L}\SS^3(I)$, $\mathcal{L}\SS^3(-I)$ or $\mathcal{L}\SS^3(Q_0)$.
For any $Q \in \mathrm{SO}_4$, there is a natural inclusion $\tilde{\mathcal{F}}$ (to be described below) of each connected component into $\Omega(\SS^3 \times \SS^3)$.
Furthermore, the inclusions $ \mathcal{L}\SS^3(\pm \mathbf{i}, \mp \mathbf{i})  \subset \Omega(\SS^3 \times \SS^3)$ are homotopy equivalences (\cite{SS12}). 
We prove that the same does not hold for the spaces $\mathcal{L}\SS^3(-\1,\1)$ and $\mathcal{L}\SS^3(\1, -\1)_n$:

\begin{Main}\label{th6}
The inclusions 
\[ \mathcal{L}\SS^3(-\mathbf{1},\1) \subset \Omega(\SS^3 \times \SS^3), \quad \mathcal{L}\SS^3(\mathbf{1},-\1)_n \subset \Omega(\SS^3 \times \SS^3) \]
are not homotopy equivalences. Moreover

\begin{eqnarray*}
 \mathrm{dim} \;  \mathrm{H}^2 (\mathcal{L}\SS^3(-\1,\1), \R) \geq 3 \quad and \quad \mathrm{dim} \; \mathrm{H}^4(\mathcal{L}\SS^3(\1,-\1)_n, \R) \geq 4. 
 \end{eqnarray*} 
\end{Main}

In particular, $\mathcal{L}\SS^3(-\1, \1)$ and $\mathcal{L}\SS^3(\1, -\1)_n$ are not homotopy equivalent to $\mathcal{L}\SS^3(\pm \mathbf{i},\mp \mathbf{i})$.
Recall that $\mathrm{H}^2 (\Omega(\SS^3 \times \SS^3), \R) = \R^2$ and $\mathrm{H}^4 (\Omega(\SS^3 \times \SS^3), \R) = \R^3$.
The methods in this papers do not immediately yield
upper estimates for these dimensions or
results for the other two spaces 
($\mathcal{L}\SS^3(\1,\1)$ and $\mathcal{L}\SS^3(-\1,-\1)$).

\bigskip

We now proceed to construct the inclusion $\tilde{\mathcal{F}}$. We do this in greater generality, for any dimension $n \geq 2$. 

A \emph{locally convex} curve on $\SS^n$ is a curve $\gamma$ of class $C^k$ ($k \geq n$) such that $ \mathrm{det}(\gamma(t),\gamma'(t),\gamma''(t), \cdots, \gamma^{(n)}(t)) > 0 $. Given a locally convex curve $\gamma:[0,1] \rightarrow \SS^n$, we associate a \emph{Frenet frame curve} $ \mathcal{F_{\gamma}}:[0,1] \rightarrow \mathrm{SO}_{n+1}$ by applying the Gram-Schmidt orthonormalization to the $(n+1)$-vectors $(\gamma(t),\gamma'(t),\dots,\gamma^{(n)}(t))$.

\begin{definition}\label{LSnQ}
For $Q \in \mathrm{SO}_{n+1}$, ${\mathcal{L}\SS^{n}}(Q)$ is the set of all locally convex curves $\gamma:[0,1] \rightarrow \SS^n$ such that $\mathcal{F}_\gamma(0)=I$ and $\mathcal{F}_\gamma(1)=Q$.
\end{definition}

For $n \geq 2$,
let $\Pi_{n+1} : \mathrm{Spin}_{n+1} \rightarrow \mathrm{SO}_{n+1}$ be the
universal double cover.
We denote by $\mathbf{1}$ the identity element in $\mathrm{Spin}_{n+1}$, and by $-\mathbf{1}$ the unique non-trivial element in $\mathrm{Spin}_{n+1}$ such that $\Pi_{n+1}(-\mathbf{1})=I$. The Frenet frame curve $\mathcal{F_{\gamma}}:[0,1] \rightarrow \mathrm{SO}_{n+1}$ can be uniquely lifted to a continuous curve $\tilde{\mathcal{F}_{\gamma}}:[0,1] \rightarrow \mathrm{Spin}_{n+1}$ such that $\mathcal{F_{\gamma}}=\Pi_{n+1} \circ \tilde{\mathcal{F}_{\gamma}}$ and $\tilde{\mathcal{F}_{\gamma}}(0)=\mathbf{1}$. 

\begin{definition}\label{LSnz}
For $z \in \mathrm{Spin}_{n+1}$, ${\mathcal{L}\SS^{n}}(z)$ is the subset of ${\mathcal{L}\SS^{n}}(\Pi_{n+1}(z))$ for which $\tilde{\mathcal{F}_{\gamma}}(1)=z$.
\end{definition}

It turns out that $\mathcal{L}\SS^n(z)$ is always non-empty. Clearly, ${\mathcal{L}\SS^{n}}(\Pi_{n+1}(z))$ is the disjoint union of ${\mathcal{L}\SS^{n}}(z)$ and ${\mathcal{L}\SS^{n}}(-z).$

Recall that $\mathrm{Spin}_{4}$ can be identified with $\SS^3 \times \SS^3$
(see Subsection~\ref{s21}).
In particular, given $z=(z_l,z_r) \in \SS^3 \times \SS^3$
(where $l$ and $r$ just stand for \textit{left} and \textit{right})
we will denote by $\mathcal{L}\SS^3(z_l,z_r)$ the space of locally convex curves in $\SS^3$ with the initial and final lifted Frenet frame respectively $(\mathbf{1},\mathbf{1})$ and $(z_l,z_r),$ i.e.,
\[ \mathcal{L}\SS^3(z_l,z_r)=\{\gamma: [0,1] \rightarrow \SS^3 \; | \, \tilde{\mathcal{F}}_\gamma(0)=(\mathbf{1},\mathbf{1}) \; \text{and} \; \tilde{\mathcal{F}}_\gamma(1)=(z_l,z_r) \}. \]  

Though the study of the spaces of locally convex curves may seem a rather specific topic, it has attracted the attention of many researchers both for its topological richness and for its connection with other areas (for example, symplectic geometry \cite{Arn95}, differential equations \cite{BST03}, control theory \cite{RS90} and engineering \cite{Dub57}).

The study of the topology of the spaces of locally convex curves on the $2$-sphere started with Little in $1970$. He proved that the space ${\mathcal{L}\SS^{2}}(I)$ has $3$ connected components (\cite{Lit70}), that we denote by ${\mathcal{L}\SS^{n}}(\mathbf{1}),  {\mathcal{L}\SS^{n}}(-\mathbf{1})_c$ and ${\mathcal{L}\SS^{n}}(-\mathbf{1})_n$. Here ${\mathcal{L}\SS^2}(-\mathbf{1})_c$ is the component of convex curves (\cite{Fen29}) and this component is contractible (\cite{Ani98}) while ${\mathcal{L}\SS^2}(-\mathbf{1})_n$ is the component associated to non-convex curves (see figure \ref{Little} below). 
\begin{figure}[H]
\centering
\includegraphics[scale=0.3]{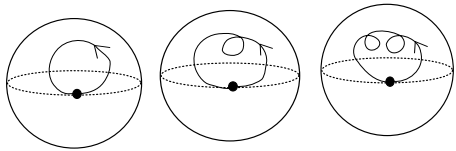}
\caption{Examples of curves in the components $\mathcal{L}\SS^2(-\mathbf{1})_c, \; \mathcal{L}\SS^2(\mathbf{1})$ and $\mathcal{L}\SS^2(\mathbf{-1})_n$, respectively.}
\label{Little}
\end{figure}

The topology of the spaces of locally convex curves on $\SS^n$ and their variations was also studied by others authors.
Among many, we mention the work of M. Z. Shapiro, B. Z. Shapiro and B. A. Khesin (\cite{SS91}, \cite{Sha93}, \cite{KS92} and \cite{KS99}) which in the 1990's determined the number of connected components of the space of locally convex curves on the $n$-sphere, in the Euclidean space, and in the Projective space.
The beautiful paper of V. I. Arnold~\cite{Arn95}
also considers related questions.
More recently, the study of Engel structures also used related methods
(\cite{CPPP15} and \cite{PP16}).
For a longer list of references, see \cite{Zul12}.

Even though the number of connected components of those spaces has been completely understood, little information on the cohomology or higher homotopy groups was available, even on the $2$-sphere. The topology of the spaces $\mathcal{L}\SS^2(\mathbf{1})$ and $\mathcal{L}\SS^2(\mathbf{-1})_n$ remained mysterious
until \cite{Sal09I}, \cite{Sal09II} and \cite{Sal13}: 

\begin{teo}[Saldanha, \cite{Sal13}]\label{thmsal}
We have the following homotopy equivalences
\[ \mathcal{L}\SS^2(\mathbf{1}) \approx (\Omega \SS^3) \vee \SS^2 \vee \SS^6 \vee \SS^{10} \vee \cdots, \quad \mathcal{L}\SS^2(-\mathbf{1})_n \approx (\Omega \SS^3) \vee \SS^4 \vee \SS^8 \vee \cdots.  \]
\end{teo}

Now we will introduce a larger space of curves that will have an important role in this work. Let $\gamma$ be a curve in $\SS^n$ of class $C^k$ ($k \geq n$): $\gamma$ is called \emph{generic} if the vectors $\gamma(t), \gamma'(t),\gamma''(t),\dots,\gamma^{(n-1)}(t)$ are linearly independent for all $t \in [0,1]$. One can still define a Frenet frame for generic curves (which are not necessarily locally convex). Indeed, one can apply Gram-Schmidt to the linearly independent vectors $\gamma(t),\gamma'(t),\dots,\gamma^{(n-1)}(t)$ to obtain $n$ orthonormal vectors $u_0(t),u_1(t), \dots, u_{n-1}(t)$. Then, there is a unique vector $u_n(t)$ for which $u_0(t),u_1(t), \dots, u_{n-1}(t), u_n(t)$ is a positive orthonormal basis. So, the continuous curve $\mathcal{F_{\gamma}}:[0,1] \rightarrow \mathrm{SO}_{n+1}$ defined by $\mathcal{F_{\gamma}}(t)=(u_0(t),u_1(t), \dots, u_{n-1}(t), u_n(t))$ is called the \emph{Frenet frame curve} of the generic curve $\gamma:[0,1] \rightarrow \SS^n$. 

\begin{definition}\label{GSnQ}
For $Q \in \mathrm{SO}_{n+1}$, $\mathcal{G}\SS^n(Q)$ is the space of all generic curves $\gamma:[0,1] \rightarrow \SS^{n}$ such that $\mathcal{F_{\gamma}}(0)=I$ and $\mathcal{F_{\gamma}}(1)=Q$. For $z \in \mathrm{Spin}_{n+1}$, ${\mathcal{G}\SS^{n}}(z)$ is the subset of ${\mathcal{G}\SS^{n}}(\Pi_{n+1}(z))$ for which $\tilde{\mathcal{F}_{\gamma}}(1)=z$.
\end{definition}

 We thus have $\mathcal{L}\SS^n(Q) \subset \mathcal{G}\SS^n(Q)$ and $\mathcal{L}\SS^n(z) \subset \mathcal{G}\SS^n(z)$.

\medskip

The homotopy type of the spaces $\mathcal{G}\SS^{n}(z)$, $z \in \mathrm{Spin}_{n+1}$, is well understood. Indeed, let us define $\Omega \mathrm{Spin}_{n+1}(z)$ to be the space of all continuous curves $\alpha : [0,1] \rightarrow \mathrm{Spin}_{n+1}$ with $\alpha(0)=\mathbf{1}$ and $\alpha(1)=z$.  It is well-known that different values of $z \in \mathrm{Spin}_{n+1}$ give rise to homeomorphic spaces $\Omega \mathrm{Spin}_{n+1}(z)$, therefore we can drop $z$ from the notation and write $\Omega \mathrm{Spin}_{n+1}$. Using the Frenet frame, we define the following \emph{Frenet frame injection} $\tilde{\mathcal{F}}: \mathcal{G}\SS^{n}(z) \rightarrow \Omega \mathrm{Spin}_{n+1}$ defined by $(\tilde{\mathcal{F}}(\gamma))(t) = \tilde{\mathcal{F}}_\gamma(t)$. The inclusion $\tilde{\mathcal{F}}:\mathcal{G}\SS^n(z) \rightarrow \Omega \mathrm{Spin}_{n+1}$ is a homotopy equivalence: this follows from the results of Hirsch and Smale (\cite{Hir59} and \cite{Sma59b}) or from the h-principle (\cite{Gro86} and \cite{Eli02}); see Subsection~\ref{s83} for a self-contained explanation.

In \cite{SS12}, Saldanha and Shapiro gave an explicit finite list $z_0, \cdots, z_k$ of elements of $\mathrm{Spin}_{n+1}$ such that, for any $z \in \mathrm{Spin}_{n+1}$, there is $z_j$ in that list such that $\mathcal{L}\SS^{n}(z)$ is homeomorphic to $\mathcal{L}\SS^{n}(z_j)$. Moreover, $\mathcal{L}\SS^{n}(z_0)$ and $\mathcal{G}\SS^{n}(z_0)$ are homeomorphic. Also, the inclusions $\mathcal{L}\SS^n(z_j) \subset \mathcal{G}\SS^n(z_j)$ induce surjective maps between homotopy or homology groups.

For $n = 3$, the result in \cite{SS12} says that given $(z_l,z_r) \in \SS^{3} \times \SS^3$, the space $\mathcal{L}\SS^{3}(z_l,z_r)$ is homeomorphic to
at least one of the five spaces:
\begin{equation*}\label{5space}
\quad \mathcal{L}\SS^3(\mathbf{i},-\mathbf{i}), \quad 
\mathcal{L}\SS^3(\mathbf{1},-\mathbf{1}), \quad \mathcal{L}\SS^3(-\mathbf{1},\mathbf{1}), \quad  
\mathcal{L}\SS^3(\mathbf{1},\mathbf{1}), \quad \mathcal{L}\SS^3(-\mathbf{1},-\mathbf{1}).
\end{equation*} 
The following homeomorphism also holds:
\[ \mathcal{L}\SS^3(\mathbf{i},-\mathbf{i}) \simeq \Omega (\SS^3 \times \SS^3) = \Omega \SS^3 \times \Omega\SS^3.  \]
Recall that
\begin{equation*}\label{loopspace}
H^j(\Omega (\SS^3 \times \SS^3),\R)=
\begin{cases}
0, & j \; \mathrm{odd} \\
\R^{l+1}, & j=2l, \; l\in\N.
\end{cases}
\end{equation*}

We would like to determine which among these $5$ spaces are homeomorphic.
We do not know the complete answer yet but we present some results:

\begin{Main}\label{th5}
For any even integer $j \geq 1$, we have
\[ \mathrm{dim} \; H^j(\mathcal{L}\SS^3(-\mathbf{1},\1),\R) \geq 1 + \mathrm{dim} \; H^j(\mathcal{G}\SS^3 (-\mathbf{1},\1),\R), \quad 4|(j+2),  \]
\[ \mathrm{dim} \; H^j(\mathcal{L}\SS^3(\mathbf{1},-\1),\R) \geq 1 + \mathrm{dim} \; H^j(\mathcal{G}\SS^3 (\mathbf{1},-\1),\R), \quad 4|j .\]

\end{Main}

Moreover, explicit generators will be constructed.
Notice that Theorem~\ref{th6} follows directly from Theorem~\ref{th5}. 

\medskip 

We will prove that any generic curve in $\mathbb{S}^3$
can be decomposed as a pair of related generic curves in $\mathbb{S}^2$
(a generic curve in $\mathbb{S}^2$ is just an immersion);
moreover, if the curve in $\mathbb{S}^3$ is locally convex, then one of the associated curves in $\mathbb{S}^2$ is also locally convex (see Theorems~\ref{th0} and~\ref{th1}). These results are very useful because they enable us to use what is known in the case $n=2$ for the case $n=3$.

Given $\gamma \in \mathcal{G}\SS^2(z)$, let us denote by $\mathbf{t}_\gamma(t)$ the unit tangent vector of $\gamma$ at the point $\gamma(t)$, that is $\mathbf{t}_\gamma(t):=\frac{\gamma'(t)}{||\gamma'(t)||} \in \SS^2.$ 
Let $\mathbf{n}_\gamma(t)$ be the unit normal vector of $\gamma$ at the point $\gamma(t)$, that is 
$ \mathbf{n}_\gamma(t):= \gamma(t) \times \mathbf{t}_\gamma(t) $ where $\times$ is the cross-product in $\R^3$. Recall that the geodesic curvature $\kappa_\gamma(t)$ is given by $ \kappa_\gamma(t):= \frac{\mathbf{t}_\gamma'(t)\cdot\mathbf{n}_\gamma(t)}{||\gamma'(t)||}$ where $\cdot$ is the Euclidean inner product.

\begin{Main}\label{th0}
There is a homeomorphism between the space $ \mathcal{G}\SS^3(z_l,z_r)$ and the space of pairs of curves  $(\gamma_l,\gamma_r) \in \mathcal{G}\SS^2(z_l)\times \mathcal{G}\SS^2(z_r)$
satisfying the condition
\begin{equation*}\label{cond00}\tag{G}
||\gamma_l'(t)||=||\gamma_r'(t)||, \quad \kappa_{\gamma_l}(t)>\kappa_{\gamma_r}(t), \quad t \in [0,1].
\end{equation*}
\end{Main}

\begin{Main}\label{th1}
There is a homeomorphism between the space $ \mathcal{L}\SS^3(z_l,z_r)$ and the space of pairs of curves  $(\gamma_l,\gamma_r) \in \mathcal{L}\SS^2(z_l)\times \mathcal{G}\SS^2(z_r)$
satisfying the condition
\begin{equation*}\label{cond}\tag{L}
||\gamma_l'(t)||=||\gamma_r'(t)||, \quad \kappa_{\gamma_l}(t)>|\kappa_{\gamma_r}(t)|, \quad t \in [0,1].
\end{equation*}
\end{Main}

We now proceed to give a brief overview of the paper.

In Section~\ref{chapter2} we start with some algebraic preliminaries. There we recall some basic notions on the spin group and on signed permutation matrices which will be necessary to explain the Bruhat decomposition of the special orthogonal group and the lifted decomposition to the spin group. This decomposition was already an important tool in~\cite{SS12}, and it will also be very important for us.

In Section~\ref{chapter3} we present some basic notions on locally convex curves and generic curves. We also define globally convex curves, which are of fundamental importance in the study of locally convex curves. In Subsection~\ref{s35} we introduce another class of curves, the Jacobian (or holonomic) and quasi-Jacobian curves. These are nothing but a different point of view on Frenet frame curves associated to locally convex curves and generic curves. 

In Section~\ref{chapter6} we prove Theorem~\ref{th0} and Theorem~\ref{th1}, which will be crucial in the sequel. Also in this section we give some examples of these results.

Finally, Section~\ref{chapter8} is devoted to the proof of Theorem~\ref{th5}. To do this, in Subsection~\ref{s81} we will introduce a notion of ``adding a pair of spirals'' to a given curve. This notion is a slight modification of the notion of ``adding a pair of loops'' to a given curve in $\SS^2$, introduced in~\cite{Sal13}.
We will do this in order to adapt more easily
the results from~\cite{Sal13} to our case; this is possible thanks to Theorems~\ref{th0} and~\ref{th1}.
This adaptation will be done in the Subsections~\ref{s83} and \ref{s84}, while our main result, Theorem~\ref{th5}, will be proved in Subsection~\ref{s85}.  

Notice that Theorems~\ref{th0} and \ref{th1} still work in the remaining spaces
$\mathcal{L}\SS^3(\1,\1)$ and $\mathcal{L}\SS^3(-\1,-\1)$.
It is the adaptation process of results from $\SS^2$ to $\SS^3$
(explained in Section~\ref{chapter8}) that has limitations
and appears to produce only some examples of tight maps
(see Section~\ref{chapter8}).

This paper is based on the Ph.D. thesis~\cite{Alv16} of the first author,
who was advised by the second.
The authors gratefully acknowledge the financial support
of CAPES, CNPq, FAPERJ and PUC-Rio,
particularly during the first author's graduate studies.
We thank Carlos Tomei, Leonardo Navarro de Carvalho, Paul Schweitzer, Pedro Zühlke and Umberto Hryniewicz, members of the Ph.D. committee, for several valuable suggestions. 
We also thank Boris Shapiro and Victor Goulart for remarks and conversations,
and the referee for a careful and helpful report.


 \section{Basic definitions and properties}

\label{chapter2}

In this section we start with some algebraic preliminaries: first we recall some definitions and basic properties of the special orthogonal groups and the spin groups, and then we explain a decomposition of these groups (the Bruhat decomposition) into finitely many subsets which will play an important role in this work.   This is closely related to but not identical to the classical Bruhat decomposition. 

\subsection{Spin groups}\label{s21}

By definition, $n \geq 2$, the spin group $\mathrm{Spin}_{n+1}$ is the universal cover of $\mathrm{SO}_{n+1}$, and it comes with a natural projection $\Pi_{n+1} : \mathrm{Spin}_{n+1} \rightarrow \mathrm{SO}_{n+1}$ which is a double covering map. Throughout this work, the unit element in the group $\mathrm{Spin}_{n+1}$ will be denoted by $\1 \in \mathrm{Spin}_{n+1}$.  

For our purposes it will be sufficient to recall a description of $\mathrm{Spin}_{n+1}$ in the cases $n=2$ and $n=3$ and it is well known that $\mathrm{Spin}_3 \simeq \SS^3$ and $\mathrm{Spin}_4 \simeq \SS^3 \times \SS^3$.  

Let us start by identifying $\R^4$ with the algebra of quaternions $\mathbb{H}$, the set of quaternions with unit norm  can be naturally identified with $\SS^3$ and the space of imaginary quaternions (i.e., of real part $0$) is naturally identified to $\R^3$.

The canonical projection $\Pi_3: \mathrm{Spin}_3 \rightarrow \mathrm{SO}_3$ is given by $ \Pi_3(z)(h) = zh\bar{z} $ for any $h \in \R^3$. In matrix notations, this map can be defined by

\begin{equation*}
\Pi_3(a+b\i+c\j+d\k)=\begin{pmatrix} 
  a^2+b^2-c^2-d^2    & -2ad+2bc & 2ac+2bd \\ 
   2ad+2bc  &  a^2-b^2+c^2-d^2 & -2ab+2cd \\
   -2ac+2bd & 2ab+2cd & a^2-b^2-c^2+d^2 
\end{pmatrix}.
\end{equation*}

The canonical projection $\Pi_4: \mathrm{Spin}_4 \rightarrow \mathrm{SO}_4$ is given by $\Pi_4 (z_l,z_r)(q) = z_lq\bar{z_r}$ for any $q \in \R^4$. The following rather cumbersome description of $\Pi_4 : \mathrm{Spin}_4 \rightarrow \mathrm{SO}_4$ in matrix notation will be used in Lemma~\ref{Lemma2}.

\begin{equation*}
\Pi_4(a_l+b_l\i+c_l\j+d_l\k,a_r+b_r\i+c_r\j+d_r\k)=
\begin{pmatrix}
C_1 & C_2 & C_3 & C_4
\end{pmatrix}
\end{equation*}
where the columns $C_i$, for $1 \leq i \leq 4$, are given by
\begin{equation*}
C_1=
\begin{pmatrix} 
a_la_r+b_lb_r+c_lc_r+d_ld_r  \\ 
-a_lb_r+b_la_r-c_ld_r+d_lc_r  \\
-a_lc_r+b_ld_r+c_la_r-d_lb_r \\
-a_ld_r-b_lc_r+c_lb_r+d_la_r 
\end{pmatrix}
\quad C_2=
\begin{pmatrix} 
a_lb_r-b_la_r-c_ld_r+d_lc_r  \\ 
a_la_r+b_lb_r-c_lc_r-d_ld_r  \\
a_ld_r+b_lc_r+c_lb_r+d_la_r  \\
-a_lc_r+b_ld_r-c_la_r+d_lb_r  
\end{pmatrix}
\end{equation*}
\begin{equation*}
C_3=
\begin{pmatrix} 
a_lc_r+b_ld_r-c_la_r-d_lb_r  \\ 
-a_ld_r+b_lc_r+c_lb_r-d_la_r  \\
a_la_r-b_lb_r+c_lc_r-d_ld_r   \\
a_lb_r+b_la_r+c_ld_r+d_lc_r  
\end{pmatrix}
\quad C_4=
\begin{pmatrix} 
a_ld_r-b_lc_r+c_lb_r-d_la_r \\ 
a_lc_r+b_ld_r+c_la_r+d_lb_r \\
-a_lb_r-b_la_r+c_ld_r+d_lc_r \\
a_la_r-b_lb_r-c_lc_r+d_ld_r
\end{pmatrix}.
\end{equation*}

\subsection{Signed permutation matrices}\label{s22}

Let $\mathrm{S}_{n+1}$ be the group of permutations on the set of $n+1$ elements $\{1, \dots, n+1\}$. An inversion of a permutation $\pi \in \mathrm{S}_{n+1}$ is a pair $(i,j) \in \{1, \dots, n+1\}^2$ such that $i<j$ and $\pi(i)>\pi(j)$.
The number of inversions of a permutation $\pi \in \mathrm{S}_{n+1}$ is denoted by $\mathrm{inv}(\pi)$. The number of inversions is at most $n(n+1)/2$, and this number is only reached by the permutation $\rho \in \mathrm{S}_{n+1}$ defined by $\rho(i)=n+2-i$ for all $i \in \{1, \dots, n+1\}$. In other words, $\rho$ is the product of transpositions $\rho=(1\;n+1)(2\; n) ... \in \mathrm{S}_{n+1}.$

A matrix $P$ is a permutation matrix if each column and each row of $P$ contains exactly one entry equal to $1$, and the others entries are zero. Permutation matrices form a finite sub-group of $\mathrm{O}_{n+1}$. There is an obvious isomorphism between the group of permutation matrices and $\mathrm{S}_{n+1}$: to a permutation $\pi \in \mathrm{S}_{n+1}$ we can associate a permutation matrix $P_\pi=(p_{i,j})$ where $ P_\pi(e_i)=e_{\pi(i)}, $ where $e_i$ denotes the $i$-th vector of the canonical basis of $\R^{n+1}$. We also write $\mathrm{inv}(P_\pi) = \mathrm{inv}(\pi).$

More generally, a signed permutation matrix is a matrix for which each column and each row contains exactly one entry equal to $1$ or $-1$, and the others entries are zero. In the notation of Coxeter groups, the set of signed permutation matrices is $\mathrm{B}_{n+1} \subset \mathrm{O}_{n+1}$, $|\mathrm{B}_{n+1}| = 2^{n+1}(n+1)!$. Given a signed permutation matrix $P$, let $\mathrm{abs}(P)$ be the associated permutation matrix obtained by dropping the signs (put differently, the entries of $\mathrm{abs}(P)$ are the absolute values of the entries of $P$). This defines a homomorphism from $\mathrm{B}_{n+1}$ to $\mathrm{S}_{n+1}$, and we set $\mathrm{inv}(P) = \mathrm{inv}(\mathrm{abs}(P))$. 

The group of signed permutation matrices of determinant one is $\mathrm{B}_{n+1}^+=\mathrm{B}_{n+1} \cap \mathrm{SO}_{n+1}$, and it has a cardinal equal to $2^{n}(n+1)!$. 

\subsection{Bruhat decomposition}\label{s23}

Let us denote by $\mathrm{Up}^+_{n+1}$ the group of upper triangular matrices with positive diagonal entries.

\begin{definition}
Given $Q \in \mathrm{SO}_{n+1}$, we define the \emph{Bruhat cell} $\mathrm{Bru}_Q$ as the set of matrices $UQU' \in \mathrm{SO}_{n+1}$, where $U$ and $U'$ belong to $\mathrm{Up}^+_{n+1}$. 
\end{definition}

Each Bruhat cell contains a unique signed permutation matrix $P \in \mathrm{B}_{n+1}^+$, hence two Bruhat cells associated to two different signed permutation matrices are disjoint. We summarize this in the following result.

\begin{proposition}[Bruhat decomposition for $\mathrm{SO}_{n+1}$]\label{Bruhat1}
We have the decomposition
\[ \mathrm{SO}_{n+1}=\bigsqcup_{P \in \mathrm{B}_{n+1}^+}\mathrm{Bru}_P. \]
\end{proposition}   

Therefore there are $2^n(n+1)!$ different Bruhat cells. Each Bruhat cell $\mathrm{Bru}_P$ is diffeomorphic to $\R^{\mathrm{inv}(P)}$, hence they are open if and only if they have maximal dimension, that is, if they correspond to the permutation $\rho$ we previously defined by $\rho=(1\;n+1)(2\; n) ... $. 

\medskip

The Bruhat decomposition of $\mathrm{SO}_{n+1}$ can be lifted to the universal double cover $\Pi_{n+1} : \mathrm{Spin}_{n+1} \rightarrow \mathrm{SO}_{n+1}$. Let us define the following sub-group of $\mathrm{Spin}_{n+1}$: 
\[ \mathrm{\tilde{B}}_{n+1}^+:=\Pi_{n+1}^{-1}(\mathrm{B}_{n+1}^+). \]
The cardinal of $\mathrm{\tilde{B}}_{n+1}^+$ is twice the cardinal of $\mathrm{B}_{n+1}^+$, that is $2^{n+1}(n+1)!$. 

\begin{definition}
Given $z \in \mathrm{Spin_{n+1}}$ we define the \emph{Bruhat cell} $\mathrm{Bru}_z$ as the connected component of $\Pi_{n+1}^{-1}(\mathrm{Bru}_{\Pi_{n+1}(z)})$ which contains $z$. 
\end{definition}

It is clear,  from the definition of $\Pi_{n+1}$, that $\Pi_{n+1}^{-1}(\mathrm{Bru}_{\Pi_{n+1}(z)})$ is the disjoint union of $\mathrm{Bru}_z$ and $\mathrm{Bru}_{-z}$, where each set $\mathrm{Bru}_{z}$, $\mathrm{Bru}_{-z}$ is contractible and non-empty.

From Proposition~\ref{Bruhat1} we have the following result.

\begin{proposition}[Bruhat decomposition for $\mathrm{Spin}_{n+1}$]\label{Bruhat2}
We have the decomposition
\[ \mathrm{Spin}_{n+1}=\bigsqcup_{\tilde{P} \in \mathrm{\tilde{B}}_{n+1}^+}\mathrm{Bru}_{\tilde{P}}. \]
\end{proposition}   

In $\mathrm{Spin}_{n+1}$, there are $2^{n+1}(n+1)!$ disjoint Bruhat cells. Each lifted Bruhat cell $\mathrm{Bru}_{\tilde{P}}$ is still diffeomorphic to $\R^{\mathrm{inv}(P)}$, where $P=\Pi_{n+1}(\tilde{P}) \in \mathrm{B}_{n+1}^+$. 

Two matrices $Q \in \mathrm{SO}_{n+1}$ and $Q' \in \mathrm{SO}_{n+1}$ (respectively two spins $z \in \mathrm{Spin}_{n+1}$ and $z' \in \mathrm{Spin}_{n+1}$) are said to be \emph{Bruhat-equivalent} if they belong to the same Bruhat cell. 

Let us conclude by quoting Lemma~$3.1$ in~\cite{SS12}, which will be very important in this work.

\begin{proposition}\label{bruhathomeo}
If $Q \in \mathrm{SO}_{n+1}$ and $Q' \in \mathrm{SO}_{n+1}$ (respectively $z \in \mathrm{Spin}_{n+1}$ and $z' \in \mathrm{Spin}_{n+1}$) are Bruhat-equivalent, then the spaces $\mathcal{L}\SS^{n}(Q)$ and $\mathcal{L}\SS^{n}(Q')$ (respectively $\mathcal{L}\SS^{n}(z)$ and $\mathcal{L}\SS^{n}(z')$) are homeomorphic. 
\end{proposition} 

\section{Spaces of curves}

\label{chapter3}

In this section we start with some definitions and then 
we characterize locally convex curves on $\SS^2$ and on $\SS^3$ (Subsection~\ref{s33}). Finally, we will characterize the Frenet frame curve associated to a locally convex curve on $\SS^2$ and on $\SS^3$ (Subsection~\ref{s35}).

\subsection{Preliminaries}\label{s33}

In this subsection we give some new definitions about locally convex and generic curves. We will deduce some fundamental properties about these curves.

\begin{definition}\label{LSn}
We define ${\mathcal{L}\SS^{n}}$ to be the set of all locally convex curves $\gamma : [0,1] \rightarrow \SS^{n}$ such that $\mathcal{F}_\gamma(0)=I$. We define ${\mathcal{G}\SS^{n}}$ to be the set of all generic curves $\gamma : [0,1] \rightarrow \SS^{n}$ such that $\mathcal{F}_\gamma(0)=I$.
\end{definition}

Clearly, $\mathcal{L}\SS^n(Q) \subset \mathcal{L}\SS^n$ and $\mathcal{G}\SS^n(Q) \subset \mathcal{G}\SS^n$.

\medskip

We will consider that our curves are smooth, but in the construction, we will not be bothered by the loss of smoothness due to juxtaposition of curves. The class of differentiability is not important: see~\cite{SS12},~\cite{Sal13} or~\cite{Alv16} for a discussion of this technical point.

\medskip

\begin{definition}\label{defc}
A curve $\gamma : [0,1] \rightarrow \SS^n$ is called \emph{globally convex} if any hyperplane $H \subseteq \mathbb{R}^{n+1}$ intersects the image of $\gamma$ in at most $n$ points, counting with multiplicity. 
\end{definition}

We need to clarify the notion of multiplicity in this definition. First, endpoints of the curve are not counted as intersections. Then, if $\gamma(t) \in H$ for some $t \in (0,1)$, the multiplicity of the intersection point $\gamma(t)$ is the smallest integer $k \geq 1$ such that
\[ \gamma^{(j)}(t) \in H, \quad 0 \leq j \leq k-1. \] 
So the multiplicity is one if $\gamma(t) \in H$ but $\gamma'(t) \notin H$, it is two if $\gamma(t) \in H$, $\gamma'(t) \in H$ but $\gamma''(t) \notin H$, and so on. Obviously, all globally convex curves are locally convex.

\medskip

Consider a curve $\gamma \in \mathcal{G}\SS^{2}(z)$. Recall that  \[ \mathbf{t}_\gamma(t):=\frac{\gamma'(t)}{||\gamma'(t)||}, \quad \mathbf{n}_\gamma(t):= \gamma(t) \times \mathbf{t}_\gamma(t) \quad \mathrm{and} \quad \kappa_\gamma(t):= \frac{\mathbf{t}_\gamma'(t)\cdot\mathbf{n}_\gamma(t)}{||\gamma'(t)||}. \]

We then define $ \mathcal{F_{\gamma}}(t)=(\gamma(t),\mathbf{t}_\gamma(t),\mathbf{n}_\gamma(t)) \in \mathrm{SO}_3$. A generic curve $\gamma:[0,1] \rightarrow \SS^2$ is locally convex if and only if $\kappa_\gamma(t)> 0$ for all $t \in (0,1)$; for a proof, see Proposition~$18$ in~\cite{Alv16}.

\medskip

Next we will consider $\gamma$ a generic curve on $\SS^3$, that is, $\gamma(t),\gamma'(t),\gamma''(t)$ are linearly independent, so that its Frenet frame $\mathcal{F}_\gamma(t)$ can be defined: $ \mathcal{F}_\gamma(t)e_1=\gamma(t), \quad \mathcal{F}_\gamma(t)e_2=\mathbf{t}_\gamma(t)=\frac{\gamma'(t)}{||\gamma'(t)||}.  $
The unit normal $\mathbf{n}_\gamma(t)$ and binormal $\mathbf{b}_\gamma(t)$ are defined by 
\[ \mathbf{n}_\gamma(t)=\mathcal{F}_\gamma(t)e_3, \quad \mathbf{b}_\gamma(t)=\mathcal{F}_\gamma(t)e_4  \]
so that  $ \mathcal{F_{\gamma}}(t)=(\gamma(t),\mathbf{t}_\gamma(t),\mathbf{n}_\gamma(t),\mathbf{b}_\gamma(t)) \in \mathrm{SO}_4. $
The geodesic curvature $\kappa_\gamma(t)$ and the geodesic torsion $\tau_\gamma(t)$ are given by:
\[ \kappa_\gamma(t):=\frac{\mathbf{t}_\gamma'(t)\cdot\mathbf{n}_\gamma(t)}{||\gamma'(t)||},  \quad \tau_\gamma(t):=\frac{-\mathbf{b}_\gamma'(t)\cdot\mathbf{n}_\gamma(t)}{||\gamma'(t)||}. \]
The geodesic curvature is never zero for generic curves. We can then characterize locally convex curves in $\SS^3$: a generic curve $\gamma:[0,1] \rightarrow \SS^3$ is locally convex if and only if $\tau_\gamma(t)> 0$ for all $t \in (0,1)$; for a proof, see Proposition~$19$ in~\cite{Alv16}.

\begin{example}\cite{SS12}\label{ex2}
Consider the curve $\xi :[0,1] \rightarrow \SS^{n}$ defined as follows. \\ For $n+1=2k$, take positive numbers $c_1,\dots, c_k$ such that $c_1^2+\cdots + c_k^2=1$ and $a_1, \dots, a_k>0$ mutually distinct, and set  
\[ \xi(t)=(c_1\cos(a_1t),c_1\sin(a_1t),\dots,c_k\cos(a_kt),c_k\sin(a_kt)).  \]
Similarly, for $n+1=2k+1$, set  
\[ \xi(t)=(c_0,c_1\cos(a_1t),c_1\sin(a_1t),\dots,c_k\cos(a_kt),c_k\sin(a_kt)). \]
In both cases, the fact that the curve $\xi$ is locally convex follows from a simple computation. 
\end{example}

In the case $n=3$, a locally convex curve looks like an ancient phone wire (see the Figure~\ref{fig:a} below).

\begin{figure}[h]
\centering
\includegraphics[scale=0.3]{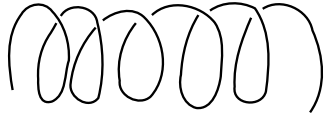}
\caption{An ancient phone wire is locally convex in $\SS^3$.}
\label{fig:a}
\end{figure}

\subsection{Holonomic and quasi-holonomic curves}\label{s35}

We will be interested in characterizing the Frenet frame curve associated to a locally convex curve. Consider a curve $ \Gamma : [0,1] \rightarrow \mathrm{SO}_{n+1} $
and define its logarithmic derivative $\Lambda(t)$ by
$ \Lambda(t)=(\Gamma(t))^{-1}\Gamma'(t) $, that is,  $\Gamma'(t)=\Gamma(t)\Lambda(t). $ Since $\Gamma$ takes values in $\mathrm{SO}_{n+1}$, $\Lambda$ takes values in its Lie algebra, that is, $\Lambda(t)$ is a skew-symmetric matrix for all $t \in [0,1]$. 

When $\Gamma=\mathcal{F}_\gamma$ is the Frenet frame curve of a locally convex curve, its logarithmic derivative $\Lambda(t)$ is not an arbitrary skew-symmetric matrix. For instance, if $\gamma : [0,1] \rightarrow \SS^2$ is locally convex, then 
\[ \mathcal{F_{\gamma}}(t)=(\gamma(t),\mathbf{t}_\gamma(t),\mathbf{n}_\gamma(t)) \in \mathrm{SO}_3 \]
and by simple computations one obtains
\begin{equation}\label{logderives2}
\Lambda_\gamma(t)=(\mathcal{F}_\gamma(t))^{-1}\mathcal{F}_\gamma'(t)=
\begin{pmatrix}
0 & -||\gamma'(t)|| & 0 \\
||\gamma'(t)|| & 0 & -||\gamma'(t)||\kappa_\gamma(t) \\
0 & ||\gamma'(t)||\kappa_\gamma(t) & 0
\end{pmatrix}.
\end{equation}
In the same way, if $\gamma : [0,1] \rightarrow \SS^3$ is locally convex, then 
\[ \mathcal{F_{\gamma}}(t)=(\gamma(t),\mathbf{t}_\gamma(t),\mathbf{n}_\gamma(t),\mathbf{b}_\gamma(t)) \in \mathrm{SO}_4 \]
and one gets
\begin{equation}\label{logderives3}
\Lambda_\gamma(t)=
\begin{pmatrix}
0 & -||\gamma'(t)|| & 0  & 0 \\
||\gamma'(t)|| & 0 & -||\gamma'(t)||\kappa_\gamma(t) & 0  \\
0 & ||\gamma'(t)||\kappa_\gamma(t) & 0 & -||\gamma'(t)||\tau_\gamma(t) \\
0 & 0 & ||\gamma'(t)||\tau_\gamma(t) & 0
\end{pmatrix}.
\end{equation}
This is in fact a general phenomenon. Let us define the set $\mathfrak{J} \subset \mathfrak{so}_{n+1}$ of Jacobi matrices, i.e., tridiagonal skew-symmetric matrices with positive subdiagonal entries, in other words, matrices of the form
\[ \begin{pmatrix}
0 & -c_1 & 0 & \ldots & 0 \\
c_1 & 0 & -c_2 &  & 0 \\
 & \ddots & \ddots & \ddots &  \\
0 &  & c_{n-1} & 0 & -c_n  \\
0 &  &  0 & c_n & 0
\end{pmatrix}, \quad c_1>0, \dots, c_n>0. \]

\begin{definition}\label{jacobian}
A curve $\Gamma : [0,1] \rightarrow \mathrm{SO}_{n+1}$ is \emph{Jacobian} if its logarithmic derivative $\Lambda(t)=(\Gamma(t))^{-1}\Gamma'(t)$ belongs to $\mathfrak{J}$ for all $t \in [0,1]$. 
\end{definition}

The interest of this definition is that Jacobian curves characterize Frenet frame curves of locally convex curves. Indeed, we have the following proposition.

\begin{proposition}\label{propjacobian}
Let $\Gamma : [0,1] \rightarrow \mathrm{SO}_{n+1}$ be a smooth curve with $\Gamma(0)=I$. Then $\Gamma$ is Jacobian if and only if there exists $\gamma \in \mathcal{L}\SS^n$ such that $\mathcal{F}_\gamma=\Gamma$.
\end{proposition}

This is exactly the content of Lemma $2.1$ in~\cite{SS12}, to which we refer for a proof. Hence there is a one-to-one correspondence between locally convex curves in $\mathcal{L}\SS^n$ and Jacobian curves starting at the identity: if $\gamma \in \mathcal{L}\SS^n$, its Frenet frame curve is such a Jacobian curve, and conversely, if $\Gamma$ is a Jacobian curve with $\Gamma(0)=I$, then if we define $\gamma_\Gamma$ by setting $\gamma_\Gamma(t)=\Gamma(t)e_1$ then $\gamma_\Gamma \in \mathcal{L}\SS^n$.

Now consider a smooth curve $\Lambda : [0,1] \rightarrow \mathfrak{J}$. Then $\Lambda$ is the logarithmic derivative of a Jacobian curve $\Gamma : [0,1] \rightarrow \mathrm{SO}_{n+1}$ if and only if $\Gamma$ solves 
\[ \Gamma'(t)=\Gamma(t)\Lambda(t). \]
If $\Gamma$ solves the above equation, then so does $Q\Gamma$, for $Q \in \mathrm{SO}_{n+1}$, since the logarithmic derivative of $\Gamma$ and $Q\Gamma$ are equal. But the initial value problem 
\[ \Gamma'(t)=\Gamma(t)\Lambda(t), \quad \Gamma(0)=I \]
has a unique solution. Thus, given a curve $\Lambda : [0,1] \rightarrow \mathfrak{J}$, there is a unique curve $\gamma \in \mathcal{L}\SS^n$ such that $\Lambda_\gamma(t)=\mathcal{F}_\gamma(t)^{-1}\mathcal{F}'_\gamma(t)=\Lambda(t)$.

Consider the locally convex curve $\xi : [0,1] \rightarrow \SS^n$ defined in Example~\ref{ex2}. It is easy to see that the logarithmic derivative $\Lambda_\xi(t)$ is constant. From what we explained, any other curve which has constant logarithmic derivative has to be of the form $Q\xi$, for some $Q \in \mathrm{SO}_{n+1}$. More precisely, given any matrix $\Lambda \in \mathfrak{J}$, the map
\[ \Gamma_\Lambda(t)=\exp(t\Lambda) \in \mathrm{SO}_{n+1} \]
is a Jacobian curve whose logarithmic derivative is constant equal to $\Lambda$. The curve $\gamma_\Lambda$ defined by $\gamma_\Lambda(t)=\Gamma_\Lambda(t)e_1$ is then locally convex, and there exists $Q \in \mathrm{SO}_{n+1}$ such that $\gamma_\Lambda=Q\xi$.

\medskip

Now the Frenet frame curve $\mathcal{F}_\gamma : [0,1] \rightarrow \mathrm{SO}_{n+1}$ of $\gamma \in \mathcal{L}\SS^n$ can be lifted to a curve
\[ \tilde{\mathcal{F}}_\gamma : [0,1] \rightarrow \mathrm{Spin}_{n+1}, \]
that is $\mathcal{F}_\gamma = \tilde{\mathcal{F}}_\gamma \circ \Pi_{n+1}$ where $\Pi_{n+1} : \mathrm{Spin}_{n+1}\rightarrow \mathrm{SO}_{n+1}$ is the universal cover projection. Such a lifted Frenet frame curve $\tilde{\mathcal{F}}_\gamma$ is thus characterized by the following definition.

\begin{definition}\label{holonomic}
A curve $\tilde{\Gamma} : [0,1] \rightarrow \mathrm{Spin}_{n+1}$ is \emph{holonomic} if the projected curve $\Gamma=\tilde{\Gamma} \circ \Pi_{n+1}$ is a Jacobian curve.
\end{definition}    

To conclude, we can also characterize the Frenet frame curve associated to a generic curve. Let us define the set $\mathfrak{Q}$ of tridiagonal skew-symmetric matrices of the form
\[ \begin{pmatrix}
0 & -c_1 & 0 & \ldots & 0 \\
c_1 & 0 & -c_2 &  & 0 \\
 & \ddots & \ddots & \ddots &  \\
0 &  & c_{n-1} & 0 & -c_n  \\
0 &  &  0 & c_n & 0
\end{pmatrix}, \quad c_1>0, \dots,c_{n-1}>0, c_n \in \R. \]
Clearly, $\mathfrak{J}$ is contained in $\mathfrak{Q}$ and we have the following definition and proposition:

\begin{definition}\label{jacobian2}
A curve $\Gamma : [0,1] \rightarrow \mathrm{SO}_{n+1}$ is \emph{quasi-Jacobian} if its logarithmic derivative $\Lambda(t)=(\Gamma(t))^{-1}\Gamma'(t)$ belongs to $\mathfrak{Q}$ for all $t \in [0,1]$. Let $\Gamma : [0,1] \rightarrow \mathrm{SO}_{n+1}$ be a smooth curve with $\Gamma(0)=I$. Then $\Gamma$ is quasi-Jacobian if and only if there exists $\gamma \in \mathcal{G}\SS^n$ such that $\mathcal{F}_\gamma=\Gamma$.
\end{definition}

\begin{proposition} \label{propjacobian2} A curve $\tilde{\Gamma} : [0,1] \rightarrow \mathrm{Spin}_{n+1}$ is \emph{quasi-holonomic} if the projected curve $\Gamma=\tilde{\Gamma} \circ \Pi_{n+1}$ is a quasi-Jacobian curve.
\end{proposition}

\section{Decomposition of locally convex curves on $\SS^3$}

\label{chapter6}

The goal of this section is to prove Theorem~\ref{th0}, which states that a generic curve in $\SS^3$ can be decomposed as a pair of immersions in $\SS^2$. When restricted to locally convex curves, this gives Theorem~\ref{th1} which states that a locally convex curve in $\SS^3$ can be decomposed as a pair of curves in $\SS^2$, one of which is locally convex and the other is an immersion. This theorem will be proved in Subsection~\ref{s61}. We then give some examples (Subsection~\ref{s62}) illustrating this general procedure for locally convex curves.

\subsection{Proof of Theorem~\ref{th0} and Theorem~\ref{th1}}\label{s61}

Consider $\gamma \in \mathcal{G}\SS^3$ and its associated Frenet and lifted Frenet frame curve
\[ \mathcal{F}_\gamma : [0,1] \rightarrow \mathrm{SO}_4, \quad \tilde{\mathcal{F}}_\gamma : [0,1] \rightarrow \SS^3 \times \SS^3. \]
These are respectively quasi-Jacobian and quasi-holonomic curves, and we recall that any quasi-Jacobian or quasi-holonomic curve is of this form. Hence, characterizing generic curves in $\SS^3$ is the same as characterizing quasi-holonomic curves $ \tilde{\Gamma} : [0,1] \rightarrow \SS^3 \times \SS^3. $
Recall that the Lie algebra of $\SS^3$, viewed as the group of unit quaternions, is the vector space of imaginary quaternions 
\[ \mathrm{Im}\mathbb{H}:=\{b\i+c\j+d\k \; | \; (b,c,d) \in \R^3\} \]
and hence the Lie algebra of $\SS^3 \times \SS^3$ is the product $\mathrm{Im}\mathbb{H} \times \mathrm{Im}\mathbb{H}$. The logarithmic derivative of $\tilde{\Gamma}$ belongs to the Lie algebra of $\SS^3 \times \SS^3$, that is
\[ \Lambda_{\tilde{\Gamma}}(t)=\tilde{\Gamma}(t)^{-1}\tilde{\Gamma}'(t) \in \mathrm{Im}\mathbb{H} \times \mathrm{Im}\mathbb{H}, \quad t \in [0,1].  \]   
In the proposition below, we characterize the subset of $\mathrm{Im}\mathbb{H} \times \mathrm{Im}\mathbb{H}$ to which the logarithmic derivative of a quasi-holonomic curve belongs. Let \[ \tilde{\mathfrak{Q}}:=\{(b_l\i+d\k,b_r\i+d\k) \in \mathrm{Im}\mathbb{H} \times \mathrm{Im}\mathbb{H} \; | \; (b_l,b_r,d) \in \R^3, \; b_l>b_r, \; d>0 \}. \]

\begin{proposition}\label{propth1}
Let $\tilde{\Gamma} : [0,1] \rightarrow \SS^3 \times \SS^3$ be a smooth curve with $\tilde{\Gamma}(0)=(\1,\1)$. Then $\tilde{\Gamma}$ is quasi-holonomic if and only if its logarithmic derivative satisfies 
\[ \Lambda_{\tilde{\Gamma}}(t) \in \tilde{\mathfrak{Q}}, \quad t \in [0,1]. \]
Moreover, if $ \Lambda_{\tilde{\Gamma}}(t)=(b_l(t)\i+d(t)\k,b_r(t)\i+d(t)\k) \in \tilde{\mathfrak{Q}}, \quad t \in [0,1], $ then 
\[ b_l(t)-b_r(t)=||\gamma'(t)||, \quad 2d(t)=||\gamma'(t)||\kappa_\gamma(t), \quad b_l(t)+b_r(t)=||\gamma'(t)||\tau_\gamma(t)\]
where the curve $\gamma : [0,1] \rightarrow \SS^3$ is defined by $ \gamma(t)=(\Pi_4 \circ \tilde{\Gamma}(t))e_1. $
\end{proposition}

\begin{proof}
By definition, $\tilde{\Gamma}$ is quasi-holonomic if and only if the projected curve
\[ \Gamma=\Pi_4 \circ \tilde{\Gamma} : [0,1] \rightarrow \mathrm{SO}_4 \]
is quasi-Jacobian, and by definition, $\Gamma$ is quasi-Jacobian if only if its logarithmic derivative belongs to the subset $\mathfrak{Q}$ of matrices of the form
\[ \begin{pmatrix}
0 & -c_1 & 0 & 0  \\
c_1 & 0 & -c_2 & 0 \\
0 & c_2 & 0 & -c_3  \\
0 & 0 & c_3 & 0 
\end{pmatrix}, \quad c_1>0, c_2>0, c_3 \in \R. \]
By the chain rule we have $ \Gamma'(t)=(D_{\tilde{\Gamma}(t)}\Pi_4) \tilde{\Gamma}'(t) $
hence 
\[ \Lambda_{\Gamma}(t)=\Gamma(t)^{-1}\Gamma'(t)=\Gamma(t)^{-1}(D_{\tilde{\Gamma}(t)}\Pi_4) \tilde{\Gamma}'(t)=\Gamma(t)^{-1}(D_{\tilde{\Gamma}(t)}\Pi_4)\tilde{\Gamma}(t)\Lambda_{\tilde{\Gamma}}(t). \]
But since $\Gamma(t)^{-1}(D_{\tilde{\Gamma}(t)}\Pi_4)\tilde{\Gamma}(t)$ is the differential of $\Pi_4$ at the identity $(\1,\1)$, we obtain $ \Lambda_{\Gamma}(t)=(D_{(\1,\1)}\Pi_4) \Lambda_{\tilde{\Gamma}}(t) $
hence to prove the first part of the proposition, one needs to prove that $\mathfrak{Q}=D_{(\1,\1)}\Pi_4(\tilde{\mathfrak{Q}})$. The differential 
\[ D_{(\1,\1)}\Pi_4 : \mathrm{Im}\mathbb{H} \times \mathrm{Im}\mathbb{H} \rightarrow \mathfrak{so}_4 \]
is given by $ D_{(\1,\1)}\Pi_4(h_l,h_r) : z \in \H \mapsto h_lz-zh_r \in \H $ for $(h_l,h_r) \in \mathrm{Im}\mathbb{H} \times \mathrm{Im}\mathbb{H}$. If we let $ h_l=b_l\i+c_l\j+d_l\k, \quad h_r=b_r\i+c_r\j+d_r\k  $
then 
\[ D_{(\1,\1)}\Pi_4(h_l,h_r)z=b_l\i z+c_l\j z+d_l\k z-(b_rz\i+c_rz\j+d_rz\k). \]
Let us denote by $\i_l$, $\j_l$ and $\k_l$ the matrices in $\mathfrak{so}_4$ that correspond to left multiplication by respectively $\i$, $\j$ and $\k$; similarly we define $\i_r$, $\j_r$ and $\k_r$ the matrices in $\mathfrak{so}_4$ that correspond to right multiplication by respectively $\bar{\i}$, $\bar{\j}$ and $\bar{\k}$. These matrices are given by 
\[      
\i_l=
\begin{pmatrix}
0 & -1 & 0 & 0 \\
+1 & 0 & 0 & 0 \\
0& 0 & 0  & -1 \\
0 & 0 & +1 & 0
\end{pmatrix},
\quad
\i_r=
\begin{pmatrix}
0 & +1 & 0 & 0 \\
-1 & 0 & 0 & 0 \\
0 & 0 & 0  & -1 \\
0 & 0 & +1 & 0
\end{pmatrix},
\]

\[
\j_l=
\begin{pmatrix}
0 & 0 & -1 & 0 \\
0 & 0 & 0 & +1 \\
+1 & 0 & 0  & 0 \\
0 & -1 & 0 & 0
\end{pmatrix},
\quad 
\j_r=
\begin{pmatrix}
0 & 0 & +1 & 0 \\
0 & 0 & 0 & +1 \\
-1 & 0 & 0  & 0 \\
0 & -1 & 0 & 0
\end{pmatrix},
\]
\[
\k_l=
\begin{pmatrix}
0 & 0 & 0 & -1 \\
0 & 0 & -1 & 0 \\
0 & +1 & 0  & 0 \\
+1 & 0 & 0 & 0
\end{pmatrix},
\quad
\k_r=
\begin{pmatrix}
0 & 0 & 0 & +1 \\
0 & 0 & -1 & 0 \\
0 & +1 & 0  & 0 \\
-1 & 0 & 0 & 0
\end{pmatrix}.
\]
We can then express $D_{(\1,\1)}\Pi_4(h_l,h_r)$ in matrix notation:
\begin{equation*} 
D_{(\1,\1)}\Pi_4(h_l,h_r)=
\begin{pmatrix}
0 & -(b_l-b_r)& -(c_l-c_r) & -(d_l-d_r) \\
b_l-b_r & 0 & -(d_l+d_r) &  -(-c_l-c_r) \\
c_l-c_r& d_l+d_r  & 0 & -(b_l+b_r)\\
d_l-d_r & -c_l-c_r & b_l+b_r  & 0
\end{pmatrix}.
\end{equation*}
From this expression, it is clear that $(h_l,h_r) \in \tilde{\mathfrak{Q}}$ if and only if $D_{(\1,\1)}\Pi_4(h_l,h_r) \in \mathfrak{Q}$. This proves the equality $\mathfrak{Q}=D_{(\1,\1)}\Pi_4(\tilde{\mathfrak{Q}})$, and hence the first part of the proposition.

Concerning the second part of the proposition, if
\[ \Lambda_{\tilde{\Gamma}}(t)=(b_l(t)\i+d(t)\k,b_r(t)\i+d(t)\k) \in \tilde{\mathfrak{Q}}, \quad t \in [0,1], \] 
then $\Lambda_{\Gamma}(t)=D_{(\1,\1)}\Pi_4(\Lambda_{\tilde{\Gamma}}(t))$ is equal to
\begin{equation*} 
\begin{pmatrix}
0 & -(b_l(t)-b_r(t))& 0 & 0 \\
b_l(t)-b_r(t) & 0 & -2d(t) &  0 \\
0 & 2d(t)  & 0 & -(b_l(t)+b_r(t))\\
0 & 0 & b_l(t)+b_r(t)  & 0
\end{pmatrix}.
\end{equation*}  
But recall (see~\eqref{logderives3}, Subsection~\ref{s35}) that we also have
\begin{equation*}
\Lambda_\Gamma(t)=\Lambda_\gamma(t)=
\begin{pmatrix}
0 & -||\gamma'(t)|| & 0  & 0 \\
||\gamma'(t)|| & 0 & -||\gamma'(t)||\kappa_\gamma(t) & 0  \\
0 & ||\gamma'(t)||\kappa_\gamma(t) & 0 & -||\gamma'(t)||\tau_\gamma(t) \\
0 & 0 & ||\gamma'(t)||\tau_\gamma(t) & 0
\end{pmatrix}
\end{equation*}
where $ \gamma(t)=\Gamma(t)e_1=(\Pi_4 \circ \tilde{\Gamma}(t))e_1. $
So a simple comparison between the two expressions of $\Lambda_\Gamma(t)$ proves the second part of the proposition.
\end{proof}

\medskip

This proposition will allow us to prove Theorem~\ref{th0}.

\begin{proof}[Proof of Theorem~\ref{th0}]
Let $\gamma \in \mathcal{G}\SS^3(z_l,z_r)$. Consider its Frenet frame curve $\mathcal{F}_\gamma(t)$, its lifted Frenet frame curve $\tilde{\Gamma}(t)=\tilde{\mathcal{F}}_\gamma(t)$ and the logarithmic derivative
\[ \Lambda_{\tilde{\Gamma}}(t)=\tilde{\Gamma}(t)^{-1}\tilde{\Gamma}'(t). \]
From Subsection~\ref{s35}, we know that $\mathcal{F}_\gamma$ is quasi-Jacobian, hence $\tilde{\Gamma}=\tilde{\mathcal{F}}_\gamma$ is quasi-holonomic. Thus we can apply Proposition~\ref{propth1} and we can uniquely write
$ \Lambda_{\tilde{\Gamma}}(t)=(b_l(t)\i+d(t)\k,b_r(t)\i+d(t)\k) $
with 
\[ b_l(t)-b_r(t)=||\gamma'(t)||, \quad 2d(t)=||\gamma'(t)||\kappa_\gamma(t), \quad b_l(t)+b_r(t)=||\gamma'(t)||\tau_\gamma(t). \]
Equivalently,
\begin{equation}\label{bbd}
\begin{cases}
d(t)=||\gamma'(t)||\kappa_\gamma(t)/2, \\
b_l(t)=||\gamma'(t)||(\tau_\gamma(t)+1)/2, \\ 
b_r(t)=||\gamma'(t)||(\tau_\gamma(t)-1)/2.
\end{cases}  
\end{equation}
Let us then define the curves $ \tilde{\Gamma}_l : [0,1] \rightarrow \SS^3, \quad \tilde{\Gamma}_r : [0,1] \rightarrow \SS^3 $ by 
\[ \tilde{\Gamma}_l(0)=\mathbf{1}, \quad \tilde{\Gamma}_l(1)=z_l,\quad \Lambda_{\tilde{\Gamma}_l}(t)=b_l(t)\i+d(t)\k \in \mathrm{Im}\H, \quad \mathrm{and} \]  
\[ \tilde{\Gamma}_r(0)=\mathbf{1}, \quad \tilde{\Gamma}_r(1)=z_r,\quad \Lambda_{\tilde{\Gamma}_r}(t)=b_r(t)\i+d(t)\k \in \mathrm{Im}\H.  \]
The curves $\tilde{\Gamma}_l$ and $\tilde{\Gamma}_r$ are uniquely defined. Let us further define
\[ \Gamma_l:=\Pi_3 \circ \tilde{\Gamma}_l : [0,1] \rightarrow \mathrm{SO}_3, \quad \Gamma_r:=\Pi_3 \circ \tilde{\Gamma}_r : [0,1] \rightarrow \mathrm{SO}_3. \]
We want to compute the logarithmic derivative of $\Gamma_l$ and $\Gamma_r$. The differential of $\Pi_3$ at $\mathbf{1}$ can be computed exactly as we computed the differential of $\Pi_4$ at $(\mathbf{1},\mathbf{1})$ (in the proof of Proposition~\ref{propth1});
we have $ D_{\1}\Pi_3 : \mathrm{Im}\H \rightarrow \mathfrak{so}_3 $ and for $h=(b\i+c\j+d\k) \in \mathrm{Im}\H$, we can write in matrix notation
\begin{equation*} 
D_{\1}\Pi_3(h)=
\begin{pmatrix}
0 & -2d & -2c  \\
2d & 0 & -2b  \\
2c & 2b  & 0 \\
\end{pmatrix}.
\end{equation*}
From this expression we obtain
\begin{equation}\label{exp1}
\Lambda_{\Gamma_l}(t)=D_{\1}\Pi_3(\Lambda_{\tilde{\Gamma}_l}(t))=
\begin{pmatrix}
0 & -2d(t) & 0  \\
2d(t) & 0 & -2b_l(t)  \\
0 & 2b_l(t)  & 0 \\
\end{pmatrix}
\end{equation} 
and
\begin{equation*}\label{exp2}
\Lambda_{\Gamma_r}(t)=D_{\1}\Pi_3(\Lambda_{\tilde{\Gamma}_r}(t))=
\begin{pmatrix}
0 & -2d(t) & 0  \\
2d(t) & 0 & -2b_r(t)  \\
0 & 2b_r(t)  & 0 \\
\end{pmatrix}. 
\end{equation*}
From~\eqref{bbd}, we see that $d(t)>0$ and $b_l(t) \in \R$, hence $\Gamma_l$ is a quasi-Jacobian curve, and therefore if we define $ \gamma_l(t):=\Gamma_l(t)e_1 $
then $\gamma_l \in \mathcal{G}\SS^2(z_l)$. Moreover, recall from~\eqref{logderives2}, Subsection~\ref{s35}, that
\begin{equation*}
\Lambda_{\Gamma_l}(t)=\Lambda_{\gamma_l}(t)=
\begin{pmatrix}
0 & -||\gamma_l'(t)|| & 0  \\
||\gamma_l'(t)|| & 0 & -||\gamma_l'(t)||\kappa_{\gamma_l}(t)  \\
0 & ||\gamma_l'(t)||\kappa_{\gamma_l}(t)  & 0 \\
\end{pmatrix}
\end{equation*}
so that comparing this with~\eqref{exp1} and recalling~\eqref{bbd}, we find
\[ ||\gamma_l'(t)||=2d(t)=||\gamma'(t)||\kappa_{\gamma}(t) \quad \mathrm{and} \]
\[ \kappa_{\gamma_l}(t)=\frac{2b_l(t)}{ ||\gamma_l'(t)||}=\frac{||\gamma_l'(t)||(\tau_\gamma(t)+1)}{ ||\gamma_l'(t)||\kappa_\gamma(t)}=\frac{\tau_\gamma(t)+1}{\kappa_\gamma(t)}. \]
Now $\Gamma_r$ is also a quasi-Jacobian curve, hence if we define $ \gamma_r(t):=\Gamma_r(t)e_1, $ then $\gamma_r \in \mathcal{G}\SS^2(z_r)$, and as before, we have
\begin{equation*}
\Lambda_{\Gamma_r}(t)=\Lambda_{\gamma_r}(t)=
\begin{pmatrix}
0 & -||\gamma_r'(t)|| & 0  \\
||\gamma_r'(t)|| & 0 & -||\gamma_r'(t)||\kappa_{\gamma_r}(t)  \\
0 & ||\gamma_r'(t)||\kappa_{\gamma_r}(t)  & 0 \\
\end{pmatrix}
\end{equation*} 
\[ \mathrm{and} \quad ||\gamma_r'(t)||=||\gamma'(t)||\kappa_{\gamma}(t), \quad \kappa_{\gamma_r}(t)=\frac{\tau_\gamma(t)-1}{\kappa_\gamma(t)}. \]
This shows that given $\gamma \in \mathcal{L}\SS^3(z_l,z_r)$, there exists a unique pair of curves $(\gamma_l,\gamma_r)$, with $\gamma_l \in \mathcal{G}\SS^2(z_l)$ and $\mathcal{G}\SS^2(z_r)$ such that $ ||\gamma_l'(t)||=||\gamma_r'(t)||, \quad \kappa_{\gamma_l}(t)>\kappa_{\gamma_r}(t) $
and moreover $ ||\gamma_l'(t)||=||\gamma_r'(t)||=||\gamma'(t)||\kappa_{\gamma}(t), \quad \kappa_{\gamma_l}(t)=\frac{\tau_\gamma(t)+1}{\kappa_\gamma(t)}, \quad \kappa_{\gamma_r}(t)=\frac{\tau_\gamma(t)-1}{\kappa_\gamma(t)}.  $ This defines a map $\gamma \mapsto (\gamma_l,\gamma_r)$, which, by construction is continuous. Conversely, given a pair of curves $(\gamma_l,\gamma_r)$, with $\gamma_l \in \mathcal{G}\SS^2(z_l)$ and $\mathcal{G}\SS^2(z_r)$ such that
$ ||\gamma_l'(t)||=||\gamma_r'(t)||, \quad \kappa_{\gamma_l}(t)>\kappa_{\gamma_r}(t), $ by simply reversing the construction above, we can find a unique curve $\gamma \in \mathcal{G}\SS^3(z_l,z_r)$ such that \[ \kappa_\gamma(t)=\frac{2}{\kappa_{\gamma_l}(t)-\kappa_{\gamma_r}(t)}, \] 
\[ \tau_\gamma(t)=\frac{\kappa_\gamma(t)(\kappa_{\gamma_l}(t)+\kappa_{\gamma_r}(t))}{2}=\frac{\kappa_{\gamma_l}(t)+\kappa_{\gamma_r}(t)}{\kappa_{\gamma_l}(t)-\kappa_{\gamma_r}(t)}, \] 
\[ ||\gamma'(t)||=\frac{||\gamma_l'(t)||}{\kappa_\gamma(t)}=\frac{||\gamma_l'(t)||(\kappa_{\gamma_l}(t)-\kappa_{\gamma_r}(t))}{2}.  \]
This also defines a map $(\gamma_l,\gamma_r) \mapsto \gamma$, which is also clearly continuous, and this completes the proof of the theorem.
\end{proof}

The proof of Theorem~\ref{th1} follows directly from the statement of Theorem~\ref{th0}. Alternatively, one can proceed exactly as in the proof of Theorem~\ref{th0}, replacing quasi-holonomic curves (respectively quasi-Jacobian curves) by holonomic curves (respectively Jacobian curves), replacing $\mathfrak{Q}$ and $\tilde{\mathfrak{Q}}$ by respectively $\mathfrak{J}$ and  
\[ \tilde{\mathfrak{J}}:=\{(b_l\i+d\k,b_r\i+d\k) \in \mathrm{Im}\mathbb{H} \times \mathrm{Im}\mathbb{H} \; | \; (b_l,b_r,d) \in \R^3, \; b_l>|b_r|, \; d>0 \}. \]

A locally convex curve in $\SS^3$ is rather hard to understand from a geometrical point of view; Theorem~\ref{th1} allows us to see such a curve as a pair of curves in $\SS^2$, a situation where one can use geometrical intuition.

\subsection{Examples} \label{s62}

Before the examples let's introduce some notation that is going to be useful for what follows.

For a real number $0<c \leq 2\pi$, let $\sigma_c : [0,1] \rightarrow \SS^2$ be the unique circle of length $c$, that is $||\sigma_c'(t)||=c$, with fixed initial and final Frenet frame equals to the identity (see Figure~\ref{fig:b}). Setting $c=2\pi\sin\rho$ (where $\rho \in (0,\pi/2]$ is the radius of curvature), this curve can be given by the following formula
\[ \sigma_c(t)=\cos\rho(\cos\rho,0,\sin\rho)+\sin\rho(\sin\rho\cos(2\pi t), \sin(2\pi t), -\cos\rho\cos(2\pi t)). \] 
The geodesic curvature of this curve is given by $\cot(\rho) \in [0,+\infty)$. Given $m>0$, let us define the curve $\sigma_c^m$ as the curve $\sigma_c$ iterated $m$ times, that is 
\[ \sigma_c^m(t)=\sigma_c(mt), \quad t \in [0,1]. \]   

\begin{figure}[H]
\centering
\includegraphics[scale=0.3]{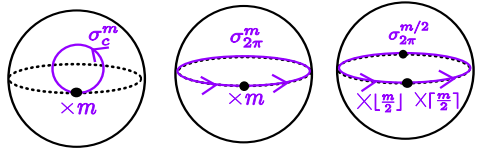}
\caption{The curves $\sigma_c^m$, $\sigma_{2\pi}^m$ and $\sigma_{2\pi}^{m/2}$.}
\label{fig:b}
\end{figure}

\begin{example}\label{family1.1}

This first example (see Figure~\ref{fig:c}) is a convex curve $\gamma_{1}^{1} \in \mathcal{L}\SS^3(-\1,\k).$ Consider $\Gamma_{1}^{1}:[0,1] \rightarrow \mathrm{SO}_{4}, \; t \mapsto  \exp(t \Lambda_{\Gamma_{1}^{1}}(t)),$ where 
\[  
\Lambda_{\Gamma_1^1}(t)=\frac{\pi}{2}
\begin{pmatrix}
0 & -\sqrt{3} & 0  & 0 \\
\sqrt{3} & 0 & -2 & 0  \\
0 & 2 & 0 & -\sqrt{3} \\
0 & 0 & \sqrt{3} & 0
\end{pmatrix}.
\]

Define $\gamma_{1}^{1}(t) := \Gamma_1^1(t)(e_1)$. 

\begin{figure}[H]
\centering
\includegraphics[scale=0.3]{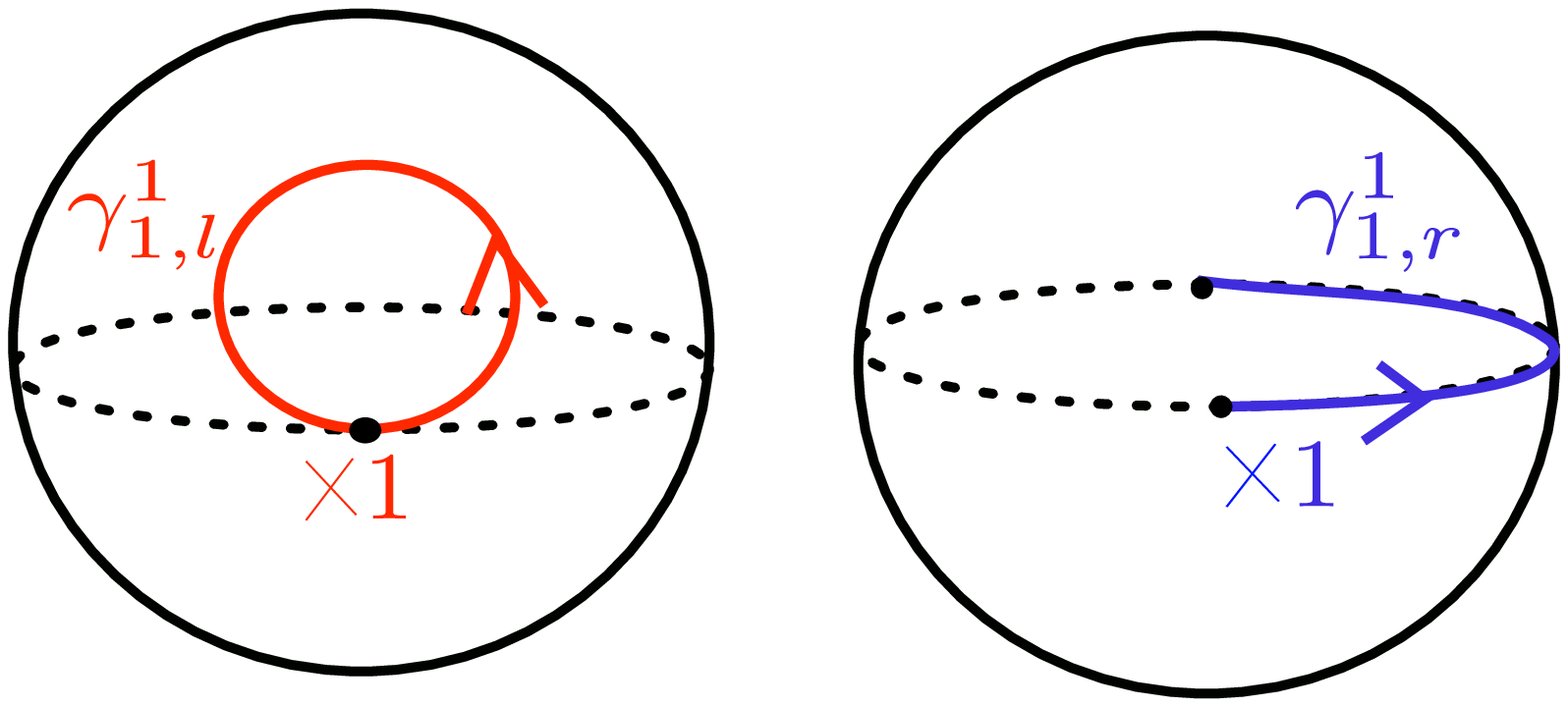}
\caption{The curve $\gamma_1^1$, where $\gamma_{1,l}^1 = \sigma_\pi^1$ and $\gamma_{1,r}^1 = \sigma_{2\pi}^{1/2}$.}
\label{fig:c}
\end{figure}

\end{example}

\begin{example}\label{family1.2}

This second example (see Figure~\ref{fig:d}) also is a convex curve; denoted by $\gamma_{1}^{2} \in \mathcal{L}\SS^3(\1,-\1).$ Consider $\Gamma_{1}^{2}:[0,1] \rightarrow \mathrm{SO}_{4}, \; t \mapsto  \exp(t \Lambda_{\Gamma_{1}^{2}}(t)),$ where 
\[  
\Lambda_{\Gamma_1^2}(t)=\frac{\pi}{2}
\begin{pmatrix}
0 & -2\sqrt{3} & 0  & 0 \\
2\sqrt{3} & 0 & -4 & 0  \\
0 & 4 & 0 & -2\sqrt{3} \\
0 & 0 & 2\sqrt{3} & 0
\end{pmatrix}.
\]

Define $\gamma_{1}^{2}(t) := \Gamma_1^2(t)(e_1)$.

\begin{figure}[H]
\centering
\includegraphics[scale=0.3]{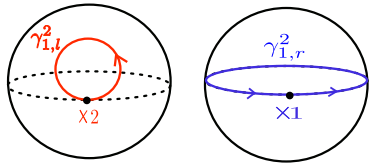}
\caption{The curve $\gamma_1^2$, where $\gamma_{1,l}^2 = \sigma_\pi^2$ and $\gamma_{1,r}^2 = \sigma_{2\pi}^{1}$.}
\label{fig:d}
\end{figure}
\end{example}

\section{Spaces $\mathcal{L}\SS^3(\1,-\1)$ and $\mathcal{L}\SS^3(-\1,\1)$}

\label{chapter8}

Recall that the spaces we are interested in are 
$ \mathcal{L}\SS^3(\1,-\1)$ and $\mathcal{L}\SS^3(-\1,\1). $
In each case, the final lifted Frenet frame does not belong to an open Bruhat cell.

Using the chopping operation, we can replace these spaces by other equivalent spaces where the final lifted Frenet frame does belong to an open Bruhat cell.

\begin{proposition}\label{spaces}
We have homeomorphisms

\medskip

$ \quad \quad \mathcal{L}\SS^3(\1,-\1) \simeq \mathcal{L}\SS^3(-\1,\k) \quad $ and $ \quad \mathcal{L}\SS^3(-\1,\1) \simeq \mathcal{L}\SS^3(\1,-\k). $

\end{proposition} 

\begin{proof}
This is an application of the chopping lemma (Proposition~$6.4$ in \cite{SS12}); see also Proposition~$70$ in~\cite{Alv16}.
\end{proof}

In the sequel, when convenient, we will look at the spaces
$ \mathcal{L}\SS^3(-\1,\k) $ and $ \mathcal{L}\SS^3(\1,-\k). $
The spins (or pair of quaternions) $(\1,-\k)$ and $(-\1,\k)$ belong to open Bruhat cells. 

In this section we prove our main result: Theorem~\ref{th5} (see Subsection~\ref{s85}). In particular, the spaces $\mathcal{L}\SS^3(-\1,\1) \simeq \mathcal{L}\SS^3(\1,-\k)$ and $\mathcal{L}\SS^3(\1,-\1) \simeq \mathcal{L}\SS^3(-\1,\k)$ are not homotopically equivalent to the space of generic curves.

\subsection{Adding loops and spirals}\label{s81}

In this subsection, we describe an operation which geometrically consists in adding a pair of loops to a generic curve in $\SS^2$, and adding a closed spiral to a generic curve in $\SS^3$. In order to avoid repeating definitions, we will describe many constructions in $\SS^n$ but we are interested in $n=2$ and $n=3$. We will study in more detail the case $n=3$ in Subsection~\ref{s83}.

For $n=2$ or $n=3$, let us fix an element $\omega_n \in \mathcal{L}\SS^n(\1)$. For $n=2$, we choose $ \omega_2=\sigma_c^2 \in \mathcal{L}\SS^2(\1) $ where $0<c<2\pi$. For $n=3$, we choose $\omega_3=\gamma_1^4 \in \mathcal{L}\SS^3(\1,\1)$,
with

\begin{eqnarray*}
\gamma_1^4(t) & = & \left(\frac{1}{4}\cos\left(6t\pi \right)+\frac{3}{4}\cos\left(2t\pi \right)\right., \;   \frac{\sqrt{3}}{4}\sin\left(6t\pi \right)+\frac{\sqrt{3}}{4}\sin\left(2t\pi \right), \\
&  &  \quad \frac{\sqrt{3}}{4}\cos\left(2t\pi \right)-\frac{\sqrt{3}}{4}\cos\left(6t\pi \right), \;
\left.\frac{3}{4}\sin\left(2t\pi\right)-\frac{1}{4}\sin\left(6t\pi\right)\right).
\end{eqnarray*}

Also the left and right part of this curve are given by
\[  \gamma_{1,l}^4=\sigma_{\pi}^4 \in \mathcal{L}\SS^2(\1), \quad \gamma_{1,r}^4 =\sigma_{2\pi}^{2} \in \mathcal{G}\SS^2(\1). \]

Coming back to the general case let us now define the operation of adding the closed curve $\omega_n$ to some curve $\gamma \in \mathcal{G}\SS^n(z)$ at some time $t_0 \in [0,1]$.

\begin{definition}\label{adding}
Take $\gamma \in \mathcal{G}\SS^n(z)$, and choose some point $t_0 \in [0,1]$. We define the curve $\gamma \ast_{t_0} \omega_n \in \mathcal{G}\SS^n(z)$ as follows. Given $\varepsilon>0$ sufficiently small, for $t_0 \in (0,1)$ we set
\begin{equation*}
\gamma \ast_{t_0} \omega_n(t)=
\begin{cases}
\gamma(t), & 0 \leq t \leq t_0-2\varepsilon \\
\gamma(2t-t_0+2\varepsilon), & t_0-2\varepsilon \leq t \leq t_0-\varepsilon \\
\mathcal{F}_\gamma(t_0)\omega_n\left(\frac{t-t_0+\varepsilon}{2\varepsilon}\right), & t_0-\varepsilon \leq t \leq t_0+\varepsilon \\
\gamma(2t-t_0-2\varepsilon), & t_0+\varepsilon \leq t \leq t_0+2\varepsilon \\
\gamma(t), & t_0+2\varepsilon \leq t \leq 1.
\end{cases}
\end{equation*}
For $t_0=0$, we set
\begin{equation*}
\gamma \ast_{0} \omega_n(t)=
\begin{cases}
\omega_n\left(\frac{t}{\varepsilon}\right), & 0 \leq t \leq \varepsilon\\
\gamma(2t-2\varepsilon), & \varepsilon \leq t \leq 2\varepsilon \\
\gamma(t), & 2\varepsilon \leq t \leq 1,
\end{cases}
\end{equation*}
and for $t_0=1$, we set
\begin{equation*}
\gamma \ast_{1} \omega_n(t)=
\begin{cases}
\gamma(t), & 0 \leq t \leq 1-2\varepsilon \\
\gamma(2t-1+2\varepsilon), & 1-2\varepsilon \leq t \leq 1-\varepsilon \\
\omega_n\left(\frac{t-1+\varepsilon}{\varepsilon}\right), & 1-\varepsilon \leq t \leq 1.
\end{cases}
\end{equation*}
\end{definition}

\begin{figure}[h]
\centering
\includegraphics[scale=0.5]{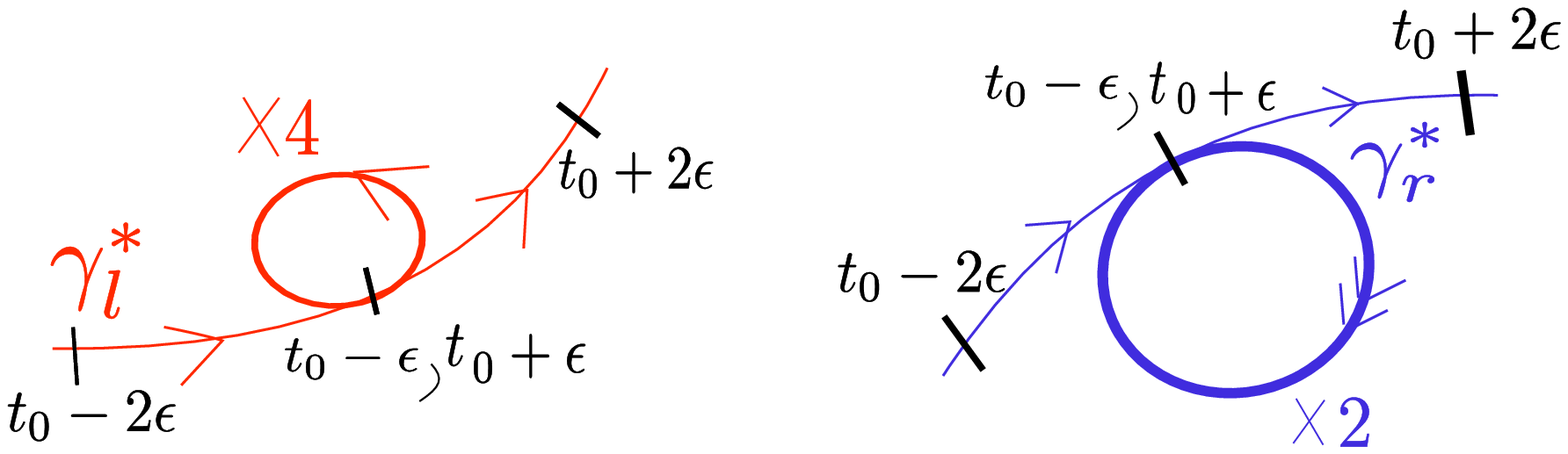}
\caption{Definition of the curve $\gamma^\ast=(\gamma_l^\ast,\gamma_r^\ast) \in \mathcal{L}\SS^3(z_l,z_r)$.}
\label{fig:m}
\end{figure}

This operation can be understood as follows (see Figure~\ref{fig:m} for an illustration in the case $n=3$). For $t_0 \in (0,1)$, we start by following the curve $\gamma$ as usual, then we speed a little slightly before $t_0$ in order to have time to insert $\omega_n$ at time $t_0$ ($\omega_n$ was moved to the correct position by a multiplication with $\mathcal{F}_\gamma(t_0)$), we speed again a little and finally at the end we follow $\gamma$ as usual. For $t_0=0$ or $t_0=1$, we have a similar interpretation. The precise value of $\varepsilon$ is not important; a different value will yield a different parametrization but the same curve. 

The precise choice of $\omega_n$ will not be important either. Indeed, the space $\mathcal{L}\SS^n(\1)$ is path-connected for $n=2$ and $n=3$ hence if we choose any other element $\omega_n' \in \mathcal{L}\SS^n(\1)$, a homotopy between $\omega_n$ and $\omega_n'$ in $\mathcal{L}\SS^n(\1)$ will give a homotopy between the curves $\gamma \ast_{t_0} \omega_n $ and $\gamma \ast_{t_0} \omega_n'$ in $\mathcal{L}\SS^n(z)$. We will see later that the homotopy class of $\gamma \ast_{t_0} \omega_n $ is the only information we will be interested in. Therefore, to simplify notations, in the sequel we will write $\gamma_{t_0}^\ast$ instead of $\gamma \ast_{t_0} \omega_n$. 

It is clear from Definition~\ref{adding} that if $\gamma \in \mathcal{L}\SS^n(z)$, then $\gamma_{t_0}^\ast \in \mathcal{L}\SS^n(z)$.  

\begin{definition}\label{tight}
Let $K$ be a compact set. A continuous map $\alpha : K \rightarrow \mathcal{L}\SS^n(z)$ is \emph{loose} if there exist a continuous map $t_0 : K \rightarrow [0,1]$ and a homotopy $A : K \times [0,1] \rightarrow \mathcal{L}\SS^n(z)$ such that for all $s \in K$: $ A(s,0)=\alpha(s), \quad A(s,1)=\alpha(s)_{t_0(s)}^\ast. $
If the map $\alpha : K \rightarrow \mathcal{L}\SS^n(z)$ is not loose, then we call it \emph{tight}.
\end{definition}

If we identify $\alpha$ with a continuous (and hence uniform) family of curves $\alpha(s) \in \mathcal{L}\SS^n(z)$, $s \in K$, then $\alpha$ is loose if each curve $\alpha(s)$ is homotopic (with a homotopy depending continuously on $s \in K$) to the curve $\alpha(s)_{t_0(s)}^\ast$, where the time $t_0(s)$ also depends continuously on $s$. Since the definition of being loose or tight just depends on the homotopy class of $\alpha(s)_{t_0(s)}^*$, it is independent of the choice of $\omega_n \in \mathcal{L}\SS^n(\1)$ when $n=2$ or $n=3$. To further simplify notation, we will often write $\gamma^\ast$ instead of $\gamma^\ast_{t_0}$ for a curve, and $\alpha^\ast$ for the family of curves $\alpha(s)_{t_0(s)}^*$ where $s$ varies in a compact set $K$.  

We are interested in finding tight maps in order to find some extra homotopy in $\mathcal{L}\SS^3(z)$ with respect to the space of generic curves. This will be explained in more detail in Subsection~\ref{s83}.
 
We have the following proposition.

\begin{proposition}\label{loosehom}
Consider two continuous maps $\alpha, \beta : K \rightarrow \mathcal{L}\SS^n(z)$, and assume that they are homotopic. Then $\alpha$ is loose if and only if $\beta$ is loose.  
\end{proposition}

\begin{proof}
Since $\alpha$ and $\beta$ are homotopic, there exists a continuous map $ H : K \times [0,1] \rightarrow \mathcal{L}\SS^n(z) $ such that for all $s \in K$:
$ H(s,0)=\alpha(s), \quad H(s,1)=\beta(s). $ Let us define  $ H^\ast : K \times [0,1] \rightarrow \mathcal{L}\SS^n(z) $ by setting, for all $(s,t) \in K \times [0,1]$: $ H^*(s,t)=(H(s,t))^*. $ This is clearly a homotopy between $\alpha^\ast$ and $\beta^\ast$. Assume that $\alpha$ is loose; we have a homotopy between $\alpha$ and $\alpha^\ast$, but since we also have a homotopy between $\beta$ and $\alpha$ and a homotopy between $\alpha^\ast$ and $\beta^\ast$, we obtain a homotopy between $\beta$ and $\beta^\ast$, hence $\beta$ is loose. Assuming that $\beta$ is loose, the exact same argument shows that $\alpha$ is loose.   
\end{proof}

Note that $\alpha^\ast$ is always loose. 

\medskip

A curve $\gamma \in \mathcal{L}\SS^n(z)$ can be identified with the image of a continuous map $\alpha : K \rightarrow \mathcal{L}\SS^n(z)$, where $K$ is a set with one element. In this way, a curve $\gamma \in \mathcal{L}\SS^n(z)$ can be either loose or tight. The following proposition is well-known (from the works of Shapiro~\cite{Sha93} and Anisov~\cite{Ani98}).

\begin{proposition}\label{tightconvex}
A curve $\gamma \in \mathcal{L}\SS^n(z)$ is tight if and only if it is convex.
\end{proposition} 

Now let us look at the case where $n=3$. Given a continuous map $\alpha : K \rightarrow \mathcal{L}\SS^3(z_l,z_r)$, one can define its left part, $\alpha_l : K \rightarrow \mathcal{L}\SS^2(z_l)$ simply by setting $\alpha_l(s)=(\alpha(s))_l$, for $s \in K$. The following proposition gives us the relation between the tightness of $\alpha$ and the tightness of its left part $\alpha_l$. 

\begin{proposition}\label{leftight}
If $\alpha : K \rightarrow \mathcal{L}\SS^3(z_l,z_r)$ is loose, then $\alpha_l : K \rightarrow \mathcal{L}\SS^2(z_l)$ is loose. As a consequence, if $\alpha_l : K \rightarrow \mathcal{L}\SS^2(z_l)$ is tight, then $\alpha : K \rightarrow \mathcal{L}\SS^3(z_l,z_r)$ is tight.   
\end{proposition}

\begin{proof}
We assume that $\alpha$ is loose. Then there exists a continuous map $ A : K \times [0,1] \rightarrow \mathcal{L}\SS^3(z_l,z_r) $ such that for all $ s \in K$:  $ A(s,0)=\alpha(s), \quad A(s,1)=\alpha(s)^\ast. $
Let us define the map 
$ A_l : K \times [0,1] \rightarrow \mathcal{L}\SS^2(z_l) $ simply by setting $A_l(s,t)=(A(s,t))_l$. Since the map giving the left part of a curve is a continuous map, $A_l$ is continuous. But now it is easy to observe that $ A_l(s,1)=\alpha_l(s)^\ast $ which proves that $\alpha_l$ is loose.
\end{proof}

\medskip

Using Propositions~\ref{tightconvex} and~\ref{leftight}, one immediately obtains the following:

\begin{proposition}\label{leftconvex}
Take $\gamma \in \mathcal{L}\SS^3(z_l,z_r)$. If $\gamma_l $ is convex, then $\gamma$ is convex.
\end{proposition}

The converse is not true in general. The curve $\gamma_1^2$ defined in Example~\ref{family1.2} is convex (\cite{Alv16}), but its left part, which is of the form $\sigma_c^2$ for some $0<c<2\pi$, is clearly not convex.

\subsection{Generalizations for the case $n=3$}\label{s83}

In order to understand the difference between the homotopy types of $\mathcal{L}\SS^3(z)$ and $\mathcal{G}\SS^3(z)$, as in \cite{Sal13}, one would like to find maps, say defined on $K=\SS^p$ for some $p \geq 1$, which are homotopic to a constant in $\mathcal{G}\SS^3(z)$ but not homotopic to a constant in $\mathcal{L}\SS^3(z)$. Indeed, if one finds such a map, this would give a non-zero element in $\pi_p(\mathcal{L}\SS^3(z))$ which is mapped to zero in $\pi_p(\mathcal{G}\SS^3(z))$. Notice that such a map is tight.

In~\cite{Sal13}, it is proven that a map $\alpha : K \rightarrow \mathcal{L}\SS^2(z)$ is always homotopic to $\alpha^\ast$ inside the space $\mathcal{G}\SS^2(z)$ (Lemma $6.1$). A similar result holds in the case $n=3$, but in order to state and prove it we need to take a small detour. Using the result in the case $n=2$, we will prove below that a map $\alpha : K \rightarrow \mathcal{L}\SS^3(z_l,z_r)$ is always homotopic, in $\mathcal{G}\SS^3(z_l,z_r)$,
to the map $\alpha$ to which we attached (curve by curve)
a pair of loops with zero geodesic torsion, that is an element in $\mathcal{G}\SS^3(\1,\1)$ with zero geodesic torsion. One could then change the definition of $\alpha^*$ so that instead of attaching an element in $\mathcal{L}\SS^3(\1,\1)$, one attaches an element in $\mathcal{G}\SS^3(\1,\1)$ with zero geodesic torsion. The obvious problem is that if $\alpha$ takes values in $\mathcal{L}\SS^3(z_l,z_r)$, this would no longer be the case of $\alpha^*$. 

To resolve this issue, recall that to an element $g \in \mathcal{G}\SS^3(\1,\1)$ with zero geodesic torsion is associated a pair of curves $(g_l,g_r) \in \mathcal{L}\SS^2(\1) \times \mathcal{G}\SS^2(\1)$     such that $\kappa_{g_l}=-\kappa_{g_r}>0$ (which follows from Theorem~\ref{th0}). Given a curve $\gamma \in \mathcal{G}\SS^3(z_l,z_r)$, let us decompose it into its left and right parts $\gamma=(\gamma_l,\gamma_r)$, and let $\gamma \ast g$ be the curve $\gamma$ to which we attached the curve $g$ at some point. Then it is easy to see that $\gamma \ast g=(\gamma_l \ast g_l,\gamma_r \ast g_r)$, that is the left (respectively right) part of $\gamma \ast g$ is obtained by attaching the left (respectively right) part of $g$ to the left (respectively right) part of $\gamma$. As we already explained, if $\gamma$ is locally convex, then $\gamma \ast g$ is not locally convex because it does not satisfy the condition on the geodesic curvature. A first attempt would be to slightly modify $g_l$ (or $g_r$) into $\tilde{g}_l$ so that the geodesic curvature condition is met; but then the condition on the norm of the speed would not be satisfied, that is $||(\gamma_l \ast \tilde{g}_l)'(t)|| \neq ||(\gamma_r \ast g_r)'(t)||$. Hence in order to satisfy both conditions at the same time, we will have to modify the whole curve in a rather subtle way. 

At the end we should obtain a curve, that we shall call $\gamma^\#$ (to distinguish from the curve $\gamma^\ast$ which we previously defined);
$\gamma^\#$ has the property that if $\gamma$ is locally convex, then so is $\gamma^\#$. Then of course one has to know how this procedure is related to the procedure of adding loops we defined. The curve $\gamma^\#$ is of course different from the curve $\gamma^\ast$, but we will see later that $\gamma$ is loose (meaning that $\gamma$ is homotopic to $\gamma^\ast$) if and only if $\gamma$ is homotopic to $\gamma^\#$; hence defining loose and tight with respect to $\gamma^\ast$ or $\gamma^\#$ is just a matter of convenience.  

We will use the Lemma below to construct the curve $\gamma^\#$.

\begin{lemma}\label{finallemma}
Consider a convex arc $\gamma: [t_{0}-2\varepsilon,t_{0}+2\varepsilon] \rightarrow \SS^2$ and positive numbers $K_0$, $K_1$, with $K_1 > \kappa_{\gamma}(t) > K_0$,
for all $t \in [t_0-2\varepsilon,t_0+2\varepsilon]$. Then given $t_{---} \in [t_0-2\varepsilon, t_{0})$ and $t_{+++} \in (t_0,t_0+2\varepsilon]$ there exist a unique arc $\nu:[t_0-2\varepsilon,t_0+2\varepsilon] \rightarrow \SS^2$ (up to reparametrization) and times $t_{--}, \; t_{++}$ with $t_{--} \in (t_{---},t_{0})$ and $t_{++} \in (t_{0},t_{+++})$ such that

\begin{equation}\label{cond0}
\nu(t)=\gamma(t), \quad t \notin [t_{---},t_{+++}], 
\end{equation}

\begin{equation}\label{cond1}
\kappa_{\nu}(t)= K_0, \quad t \in [t_{---},t_{--}] \cup  [t_{++},t_{+++}], 
\end{equation} 

\begin{equation}\label{cond2}
\kappa_{\nu}(t)= K_1, \quad t \in [t_{--},t_{++}] \quad \mathrm{and} 
\end{equation} 

\begin{equation}\label{cond3}
\int_{t_{---}}^{t_{+++}}||\gamma'(t)||dt < \int_{t_{---}}^{t_{+++}}||\nu'(t)||dt
\end{equation}

Futhermore, $t_{---}$ and $t_{+++}$ can be chosen so that there exist $t_{-}$, $t_{+}$, with $t_{-} \in (t_{--},t_{0})$ and $t_{+} \in (t_{0},t_{++})$
and

\begin{equation}\label{cond4}
\int_{t_{---}}^{t_{0}}||\gamma'(t)||dt = \int_{t_{---}}^{t_{-}}||\nu'(t)||dt,
\end{equation}

\begin{equation}\label{cond5}
\int_{t_{0}}^{t_{+++}}||\gamma'(t)||dt = \int_{t_{+}}^{t_{+++}}||\nu'(t)||dt.
\end{equation}

\end{lemma}

\begin{proof}
Construct large tangent circles of curvature $K_0$ at $\gamma(t_{+++})$ and $\gamma(t_{---})$, as in Figure~\ref{fig:lemma}. Notice that they are external to the arc. Construct a (small) circle of curvature $K_1$ tangent to the first two circles. The curve $\nu$ is obtained by following arcs of these three circles as in Figure~\ref{fig:lemma}. Convexity implies that the outside curve $\nu$ is longer than the inside curve $\gamma$. This takes care of conditions~(\ref{cond0}), (\ref{cond1}), (\ref{cond2}) and (\ref{cond3}).
Define $t_{+}$ and $t_{-}$ by equations~(\ref{cond4}) and~(\ref{cond5}). By choosing $t_{+++}$ and $t_{---}$, we can guarantee that $\nu(t_{+})$ and $\nu(t_{-})$ fall on the (smaller) circle of curvature $K_1$. We then choose $t_{++}$, $t_0$ and $t_{--}$ in the appropriate order.
\end{proof}

\begin{figure}[H]
\centering
\includegraphics[scale=0.5]{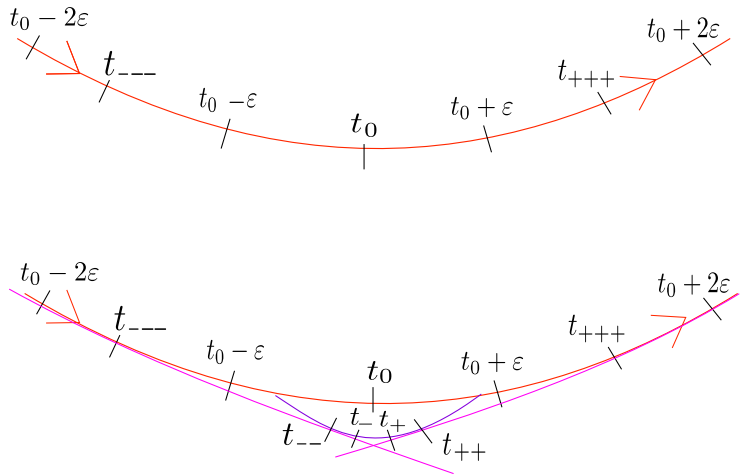}
\caption{How we modify a curve $\gamma \in \mathcal{L}\SS^2(z)$.}
\label{fig:lemma}
\end{figure}

Given a curve $\gamma \in \mathcal{L}\SS^3(z_l,z_r)$, let $\gamma_l \in \mathcal{L}\SS^2(z_l)$ and $\gamma_r \in \mathcal{G}\SS^2(z_r)$. To define the curve $\gamma^\# \in \mathcal{L}\SS^3(z_l,z_r)$, we will define its pair of curves $\gamma_l^\# \in \mathcal{L}\SS^2(z_l)$ and $\gamma_r^\# \in \mathcal{G}\SS^2(z_r)$, using the Lemma \ref{finallemma}. (The reader should follow the construction in Figure~\ref{fig:n}.) Fix $t_0 \in (0,1)$ (the cases $t_0=0$ and $t_0=1$ can be treated in a similar way). The curve we are going to define depends of course on $t_0$, but as before, we will simply write $\gamma_{t_0}^{\#}=\gamma^\#$.   

\begin{figure}[ht]
\centering
\includegraphics[scale=0.5]{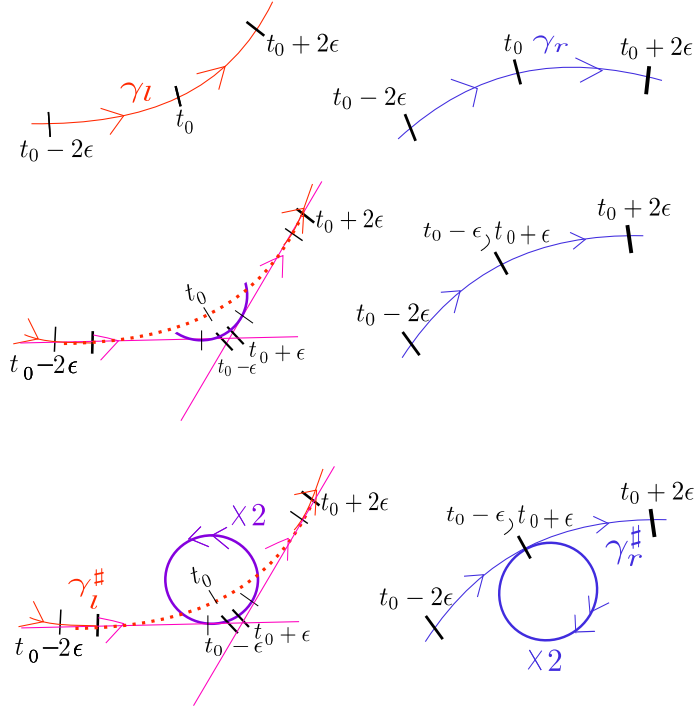}
\caption{Definition of the curve $\gamma^\#=(\gamma_l^\#,\gamma_r^\#) \in \mathcal{L}\SS^3(z_l,z_r)$.}
\label{fig:n}
\end{figure}

The curvatures of $\gamma_l$ and $\gamma_r$ at the point $t_0$ satisfy $\kappa_{\gamma_l}(t_0)>|\kappa_{\gamma_r}(t_0)|$. Since $\kappa_{\gamma_l}(t)$ and $|\kappa_{\gamma_r}(t)|$ can be assumed to be continuous, there exist $\varepsilon>0$ and $K_0>0$, $K_1>0$ such that for all $t_l \in [t_0-2\varepsilon,t_0+2\varepsilon]$ and $t_r \in [t_0-2\varepsilon,t_0+2\varepsilon]$, one has
\begin{equation}\label{curvature}
K_1> \kappa_{\gamma_l}(t_l) > K_0 > |\kappa_{\gamma_r}(t_r)|.
\end{equation} 

Now we are in the situation of the Lemma \ref{finallemma}, which we will use to construct $\gamma^{\#}=(\gamma_l^\#,\gamma_r^\#).$

Outside the interval $[t_{0}-2\varepsilon,t_{0}+2\varepsilon]$, we will not modify the curves $\gamma_l$ and $\gamma_r$, that is, we set
\begin{equation}
\gamma_l^\#(t)=\gamma_l(t), \quad \gamma_r^\#(t)=\gamma_r(t), \quad t \notin [t_0-2\varepsilon,t_0+2\varepsilon]. 
\end{equation} 
Hence for $t \notin [t_0-2\varepsilon,t_0+2\varepsilon]$,  condition (L) is clearly satisfied.

In the set $[t_{0}-2\varepsilon,t_{0}-\varepsilon] \cup [t_{0}+\varepsilon,t_{0}+2\varepsilon]$, $\gamma_r^\#$ will simply correspond to a reparametrization of $\gamma_r$, such that the curve $\gamma_r^\#$ on these intervals has two times the velocity of $\gamma_r$ in the same interval.
For $\gamma_l^\#$, $t \in [t_{0}-2\varepsilon,t_{0}-\varepsilon] \cup [t_{0}+\varepsilon,t_{0}+2\varepsilon] $ we will follow  the curve $\nu$ reparametrized by $\varphi_{-}:[t_0-2\varepsilon,t_0-\varepsilon] \rightarrow [t_0-2\varepsilon,t_{-}]$ and $\varphi_{+}:[t_0+\varepsilon,t_0+2\varepsilon] \rightarrow [t_{+},t_{0}+2\varepsilon]$. 
Therefore, from this and (\ref{cond4}) and (\ref{cond5}) the condition on the length is satisfied. The condition on the geodesic curvature is also satisfied, since in this set
\[ \kappa_{\gamma_l^\#}(t)> K_{0} > |\kappa_{\gamma_r^\#}(t)|. \]

It remains to define the curve on the interval $[t_0-\varepsilon,t_0+\varepsilon]$. Observe here that $\gamma_r^\#(t_0-\varepsilon)=\gamma_r(t_0)=\gamma_r^\#(t_0+\varepsilon)$, while $\gamma_l^\#(t_0-\varepsilon)\neq \gamma_l^\#(t_0+\varepsilon)$. Note that, by construction, $\gamma_l^\#(t_{0}-\varepsilon)=\nu(t_{-})$ and $\gamma_l^\#(t_{0}+\varepsilon)=\nu(t_{+}).$
The curve $\gamma_l^\#$ for $t \in [t_{0}-\varepsilon,t_{0}+\varepsilon]$ follows a circle of length $c_{1}$ with geodesic curvature $K_1$, performing slightly more than 2 turns. Therefore, for all $t \in [t_0-\varepsilon,t_0+\varepsilon],$ one has
\[ \kappa_{\gamma_l^\#}(t)=K_1 \quad \mathrm{and} \] 
\[ \int_{t_{0}-\varepsilon}^{t_{0}+\varepsilon} ||(\gamma_l^\#)'(t)||dt =2c_1+ \int_{t_{-}}^{t_{+}} ||(\nu'(t)||dt. \]

Choose $0<c_0<2\pi$ such that the geodesic curvature of the curve $\sigma_{c_0}$ is equal to $K_0$. 

Let us now choose $c_2=c_1+\frac{1}{2}\int_{t_{-}}^{t_{+}} ||\nu'(t)||dt$, and let $\bar{\sigma}_{c_2}$ be the curve obtained by reflecting the curve $\sigma_{c_2}$ with respect to the hyperplane $\{(x,y,z) \in \R^3 \; | \; z=0\}$ (that is, $\bar{\sigma}_{c_2}$ is the image of $\sigma_{c_2}$ under the map $(x,y,z) \mapsto (x,y,-z)$). Such a curve $\bar{\sigma}_{c_2}$ has constant negative geodesic curvature $-K_2$ (hence it is negative locally convex). Now define $\gamma_r^\#$ on $[t_0-\varepsilon,t_0+\varepsilon]$ by setting 
\begin{equation*}
\gamma_r^\#(t)=(\mathcal{F}_{\gamma_r}(t_0)) \bar{\sigma}_{c_2}^2\left(\frac{t-t_0+\varepsilon}{2\varepsilon}\right), \quad t \in [t_0-\varepsilon,t_0+\varepsilon].
\end{equation*}

So, since $c_2>c_1$ and the absolute value $K_2$ of the geodesic curvature of $\sigma_{c_2}$ satisfies $K_2<K_1$, hence
\[ \kappa_{\gamma_l^\#}(t)=K_1>K_2=|\kappa_{\gamma_r^\#}(t)|. \]
Therefore condition (L) is also satisfied on $[t_0-\varepsilon,t_0+\varepsilon]$. Again, Figure~\ref{fig:n} summarizes the definition of the curve $\gamma^\#=(\gamma_l^\#,\gamma_r^\#)$.

\begin{definition}\label{sustenido}
Given a curve $\gamma \in \mathcal{L}\SS^3(z_l,z_r)$ and a time $t_0 \in [0,1]$, we define $\gamma_{t_0}^\#=\gamma^\# \in \mathcal{L}\SS^3(z_l,z_r)$ by setting $\gamma^\#=(\gamma_l^\#,\gamma_r^\#)$, where $\gamma_l^\# \in \mathcal{L}\SS^2(z_l)$ and $\gamma_r^\# \in \mathcal{G}\SS^2(z_r)$ are defined as in the construction above. Given continuous maps $\alpha : K \rightarrow \mathcal{L}\SS^3(z_l,z_r)$ and $t_0 : K \rightarrow [0,1]$, we define $\alpha_{t_0}^\#=\alpha^\# : K \rightarrow \mathcal{L}\SS^3(z_l,z_r)$ by setting $\alpha_{t_0}^\#(s)=(\alpha(s))_{t_0(s)}^\#$ for all $s \in K$.
\end{definition}

\begin{definition}\label{tight2}
Let $K$ be a compact set. A continuous map $\alpha : K \rightarrow \mathcal{L}\SS^3(z)$ is \emph{$\#$-loose} if there exist a continuous map $ t_0 : K \rightarrow [0,1] $ and a homotopy $ A : K \times [0,1] \rightarrow \mathcal{L}\SS^3(z)$ such that for all $s \in K$: $ A(s,0)=\alpha(s), \quad A(s,1)=\alpha(s)_{t_0(s)}^\#. $
If the map $\alpha : K \rightarrow \mathcal{L}\SS^3(z)$ is not $\#$-loose, then we call it \emph{$\#$-tight}.
\end{definition}

\begin{Rem}\label{tight12}
A continuous map $\alpha : K \rightarrow \mathcal{L}\SS^3(z)$ is $\#$-loose if and only if it is loose. Therefore a continuous map $\alpha : K \rightarrow \mathcal{L}\SS^3(z)$ is $\#$-tight if and only if it is tight.
\end{Rem}
 
One can prove this remark using the techniques of ``spreading loops along a curve''  (see for instance~\cite{Sal13}), which can be seen as an easy instance of the h-principle of Gromov (\cite{Gro86} and \cite{Eli02}). We will not prove this remark since we will not use it; in the sequel it will be more convenient to deal with these concepts since they will enable us to apply more easily results in the case $n=2$.

We can now prove the following result:

\begin{proposition}\label{Genericas}
Let $\alpha: K \rightarrow \mathcal{G}\SS^3(z_l,z_r)$ be a continuous map. Then $\alpha$ is homotopic to $\alpha^\#$ inside the space $\mathcal{G}\SS^3(z_l,z_r)$. 
\end{proposition}

\begin{proof}

We know from the h-principle that the desired homotopy exists. We want, however, to have a picture of the process: this is described in detail in~\cite{Alv16}. In nutshell, given a curve $\gamma \in \mathcal{G}\SS^3(z_l,z_r)$, we define $\gamma_1$ by deforming the unit tangent vector $t_{\gamma}(t)$ as in Figure~\ref{fig:p}. This has the effect of adding two loops (with opposite orientations) to both $\gamma_l$ and $\gamma_r$, as in Figure~\ref{fig:q}.

\begin{figure}[H]
\centering
\includegraphics[scale=0.4]{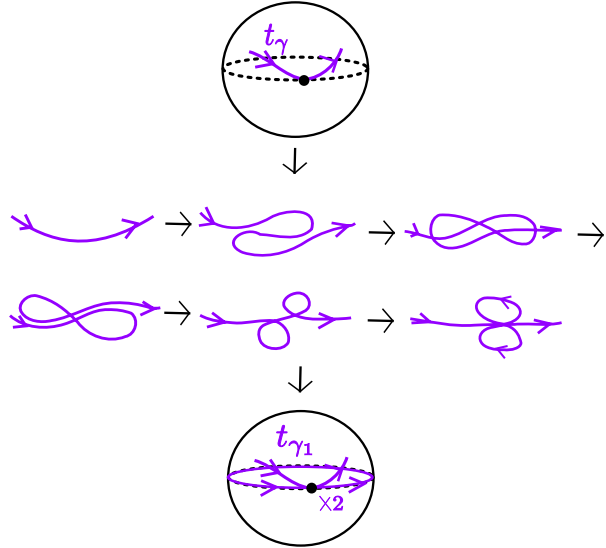}
\caption{A path from $t_{\gamma}$ to $t_{\gamma_1}$.}
\label{fig:p}
\end{figure}

\begin{figure}[H]
\centering
\includegraphics[scale=0.5]{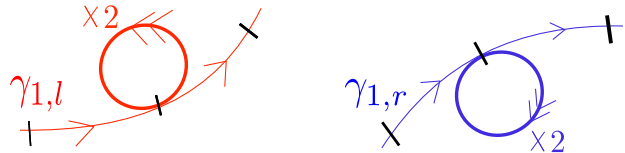}
\caption{The curves $\gamma_{1,l}$ and $\gamma_{1,r}$.}
\label{fig:q}
\end{figure}

\end{proof}

Note that Proposition~\ref{Genericas} is the analogous result for $\SS^3$
of Lemma~$6.1$ from~\cite{Sal13}.
The following remark is analogous to Proposition~$6.4$ from~\cite{Sal13}.

\begin{Rem}\label{Genericas2}
Let $\alpha: K \rightarrow \mathcal{G}\SS^3(z_l,z_r)$ be a continuous map. Then $\alpha$ is homotopic to a constant map in $\mathcal{G}\SS^3(z_l,z_r)$ if and only if $\alpha^\#$ is homotopic to a constant map in $\mathcal{L}\SS^3(z_l,z_r)$. 
\end{Rem} 

One direction follows directly from Proposition~\ref{Genericas}: if $\alpha^\#$ is homotopic to a constant map in $\mathcal{L}\SS^3(z_l,z_r)$, since $\alpha$ is always homotopic to $\alpha^\#$ in $\mathcal{G}\SS^3(z_l,z_r)$, we obtain that $\alpha$ is homotopic to a constant in $\mathcal{G}\SS^3(z_l,z_r)$. The other direction can be proved exactly as in Proposition~$6.4$ from~\cite{Sal13} using again the techniques of spreading loops along a curve.
We will not use this statement and therefore a careful proof is not given.

\medskip 

In~\cite{Sal13}, tight maps
\[ h_{2k-2} : \SS^{2k-2} \rightarrow \mathcal{L}\SS^2((-\1)^k) \]
are constructed for an integer $k \geq 2$; 
these are homotopic to constants maps in $\mathcal{G}\SS^2((-\1)^k)$.
These maps are going to be very important in this work too.
To prove that these maps are not homotopic to a constant
in $\mathcal{L}\SS^2((-\1)^k)$,
the following notion is introduced.

\begin{definition}\label{multiconvex}
A curve $\gamma \in \mathcal{L}\SS^2(z)$ is \emph{multiconvex} of multiplicity $k$ if there exist times $0 = t_0 < t_1 < \cdots <
t_k = 1$ such that $\mathcal{F}_\gamma(t_i)=I$ for $0 \leq i<k$, and the restrictions of $\gamma$ to the intervals $[t_{i-1},t_i]$ are convex arcs for $1 \leq i \leq k$. 
\end{definition}

Let us denote by $\mathcal{M}_k(z)$ the set of multiconvex curves of multiplicity $k$ in $ \mathcal{L}\SS^2(z)$.
Lemma~$7.1$ on~\cite{Sal13} proves that the set $\mathcal{M}_k(z)$
is a closed contractible submanifold of $\mathcal{L}\SS^2(z)$ of codimension $2k-2$ with trivial normal bundle. Therefore we can associate to $\mathcal{M}_k(z)$ a cohomology class $m_{2k-2} \in H^{2k-2}(\mathcal{L}\SS^2(z),\R)$ by counting intersection with multiplicity. Given any continuous map $\alpha : K \rightarrow \mathcal{L}\SS^2(z)$, by a perturbation we can make it smooth and transverse to $\mathcal{M}_k(z)$, and we denote by $m_{2k-2}(\alpha) \in \R$ the intersection number of $\alpha$ with $\mathcal{M}_k(z)$.

Therefore, $h_{2k-2}$ defines extra generators in $\pi_{2k-2}(\mathcal{L}\SS^2((-\1)^k))$ (as compared to $\pi_{2k-2}(\mathcal{G}\SS^2((-\1)^k)$) and $m_{2k-2}$ defines extra generators in $H^{2k-2}(\mathcal{L}\SS^2((-\1)^k),\R)$ (as compared to $H^{2k-2}(\mathcal{G}\SS^2((-\1)^k,\R)$). 

Our objective will be to use this extra topology given by $h_{2k-2}$ and $m_{2k-2}$ to $\mathcal{L}\SS^2((-\1)^k)$ with respect to the space of generic curves, together with our decomposition results Theorem~\ref{th0} and Theorem~\ref{th1}, to draw similar conclusions in the case $n=3$. We will be able to do this only in two cases, namely for $\mathcal{L}\SS^3(\1,-\1)$ and $\mathcal{L}\SS^3(-\1,\1)$. But first some extra work is needed.

\subsection{Relaxation-reflexion of curves in $\mathcal{L}\SS^2(\1)$ and $\mathcal{L}\SS^2(-\1)$}\label{s84}

The goal of this subsection is to address the following problem: given a continuous map $\alpha : K \rightarrow \mathcal{L}\SS^2(z_l)$, how to find a way to construct a continuous map $\hat{\alpha}: K \rightarrow \mathcal{L}\SS^3(z_l,z_r)$ such that $\hat{\alpha}_l=\alpha$. 
If we are able to do this, then we will be in a good position to use what is know in the case $n = 2$ to obtain information on our spaces of locally convex curves.

It will be sufficient to consider first the case of curves, that is given $\gamma \in \mathcal{L}\SS^2(z_l)$, we will construct $\hat{\gamma} \in \mathcal{L}\SS^3(z_l,z_r)$ such that $\hat{\gamma}_l=\gamma$. The first idea is simply to define $\hat{\gamma}_r$ to have a length equals to the length of $\hat{\gamma}_l=\gamma$, and just slightly less geodesic curvature, say the geodesic curvature of $\gamma$ reduced by a small constant $\delta$. Let us denote this curve by $R_\delta \gamma$ for the moment.

A first difficulty is that if $\hat{\gamma}=(\gamma,R_\delta \gamma)$, then the final Frenet frame of $\hat{\gamma}$ will depend on $\delta$ and also possibly on the curve $\gamma$ itself. But this is not a serious problem: instead of looking at the final Frenet frame we can look at its Bruhat cell which will be independent of $\delta$ small enough and of $\gamma$, hence after a projective transformation we may assume that the curve has a fixed final Frenet frame. 

Let us denote by $R(z_l)$ a representative of the final Frenet frame of $R_\delta \gamma$; the final Frenet frame of $\hat{\gamma}$ would then be $(z_l,R(z_l))$. Here, $z_r$ is a function of $z_l$, and many pairs $(z_l,z_r)$ are probably not reached by this procedure. But this can (and in fact will) work for the two spaces $\mathcal{L}\SS^3(-\1,\1) \simeq \mathcal{L}\SS^3(\1,-\k)$ and $\mathcal{L}\SS^3(\1,-\1) \simeq \mathcal{L}\SS^3(-\1,\k)$.  

Yet this is not sufficient. We also want this relaxation process to be compatible with the operation $\#$ defined in Subsection~\ref{s81}. More precisely, one would like to know that if $\gamma$ is such that $\gamma_r=R_\delta \gamma_l$, then $\gamma^\#$ still has this property, namely we want $\gamma^\#_r=R_\delta \gamma^\#_l$. To obtain this symmetry, we will have to relax the geodesic curvature in a symmetric way by introducing another small parameter $\varepsilon>0$, and to reflect the curve obtained: this is what we will call the relaxation-reflection of a curve $\gamma$, and it will be denoted by $RR_{\varepsilon,\delta} \gamma$. We will show that for $\gamma \in \mathcal{L}\SS^2(\1)$, this will produce a curve $\hat{\gamma}=(\gamma,RR_{\varepsilon,\delta} \gamma) \in \mathcal{L}\SS^3(\1,RR(\1)) \simeq \mathcal{L}\SS^3(\1,-\k)$ and for $\gamma \in \mathcal{L}\SS^2(-\1)$, $\hat{\gamma}=(\gamma,RR_{\varepsilon,\delta} \gamma) \in \mathcal{L}\SS^3(-\1,RR(-\1)) \simeq \mathcal{L}\SS^3(-\1,\k)$.

Let us now give proper definition.

\begin{definition}\label{relaxcurve}
Given $\gamma \in \mathcal{L}\SS^2(\pm\1)$, $\varepsilon>0$ and $\delta>0$ sufficiently small, let us define $RR_{\varepsilon,\delta} \gamma$ to be the unique curve in $\mathcal{G}\SS^2$ such that
\[ ||(RR_{\varepsilon,\delta} \gamma)'(t)||=||\gamma'(t)||, \quad\kappa_{RR_{\varepsilon,\delta} \gamma }(t)=
\begin{cases}
-\kappa_\gamma(t)+\delta, \quad t \in (0,\varepsilon) \cup (1-\varepsilon,1), \\ 
-\kappa_\gamma(t)+\delta^2\varepsilon^2, \quad t \in (\varepsilon,1-\varepsilon).
\end{cases}
\]
\end{definition}

This definition should be understood as follows. On the small union of intervals $(0,\varepsilon) \cup (1-\varepsilon,1)$, which is a symmetric interval around the initial point since our curve is closed, we relax the curvature by a constant $\delta$. On the large interval $(\varepsilon,1-\varepsilon)$, the curvature is relaxed by the much smaller constant $\delta^2\varepsilon^2$, so that the product of the relaxation of the curvature with the length of $(\varepsilon,1-\varepsilon)$, which is $\delta^2\varepsilon^2(1-2\varepsilon) \sim \delta^2\varepsilon^2$ is much smaller than the product of the relaxation of the curvature with the length of $(0,\varepsilon) \cup (1-\varepsilon,1)$, which is $2\delta\varepsilon \sim \delta\varepsilon$.   

It follows from Subsection~\ref{s35}, that this curve $RR_{\varepsilon,\delta} \gamma$ is well-defined. For $\gamma \in \mathcal{L}\SS^2(\pm \1)$, the final Frenet frame of $RR_{\varepsilon,\delta} \gamma$, which for the moment may depend upon $\varepsilon,\delta$ and $\gamma$, will be denoted by $RR_{\varepsilon,\delta,\gamma} (\pm \1)$.

Given $\gamma \in \mathcal{L}\SS^2(\pm\1)$ and $\varepsilon,\delta>0$ sufficiently small, let us define 
\[ \hat{\gamma}_{\varepsilon,\delta}=(\gamma,RR_{\varepsilon,\delta} \gamma) \in \mathcal{L}\SS^2(\pm\1) \times \mathcal{G}\SS^2(RR_{\varepsilon,\delta, \gamma} (\pm \1)).  \]

\begin{lemma}\label{Lemma2}
Consider $\gamma \in \mathcal{L}\SS^2(\pm \1)$. Then, for sufficiently small $\varepsilon, \delta > 0$, the pair $(\pm \1, RR_{\varepsilon, \delta, \gamma}(\pm \1))$ belongs to the open Bruhat cell $(\pm \1, \mp \k)$.
\end{lemma}

\begin{proof}
Let us prove that $(\1,RR_{\varepsilon,\delta,\gamma} (\1))$ is Bruhat-equivalent to $(\1,-\k)$; the proof that $(-\1,RR_{\varepsilon,\delta,\gamma} (-\1))$ is Bruhat-equivalent to $(-\1,\k)$ will be analogous.

We will first prove this for a specific curve $\gamma=\gamma_l \in \mathcal{L}\SS^2(\1)$; at the end we will explain how this implies the result for an arbitrary curve in $\mathcal{L}\SS^2(\1)$. Let us choose
\[ \gamma_l(t)=\Pi_3(\tilde{\Gamma}_l(t))(e_1), \quad t \in [0,1] \]
where 
\[ \tilde{\Gamma}_l(t)=\exp\left(2\pi h_l t\right) \in \SS^3, \quad t \in [0,1] \]
with $h_l=\cos(\theta_l)\i+\sin(\theta_l)\k$ and $\theta_l=\pi/4$, that is
\[ \tilde{\Gamma}_l(t)=\exp\left(2\pi \left(\frac{\i+\k}{\sqrt{2}}\right) t\right). \]

Then for $t \in [0,1]$ close to one, and given $\varepsilon,\delta>0$ small, we set
\[ \gamma_{r,\delta,\varepsilon}(t)=\Pi_3(\tilde{\Gamma}_{r,\delta,\varepsilon}(t))(e_1) \]
where $\tilde{\Gamma}_{r,\delta,\varepsilon}$ is defined by
\[ \tilde{\Gamma}_{r,\delta,\varepsilon}(t)=\exp\left((2\pi-\varepsilon) h_{r,\delta} t\right) \]
with $\varepsilon>0$ small and
\[ h_{r,\delta}=-\cos \theta_{r,\delta} \i+\sin\theta_{r,\delta} \k \]
with $\theta_{r,\delta}=\pi/4+\delta$, with $\delta$ small. Observe that this curve $\gamma_{r,\delta,\varepsilon}$, is not exactly the curve $RR_{\varepsilon,\delta}\gamma_l$ that we defined; yet clearly the two are homotopic hence it is enough to prove the result by considering $\gamma_{r,\delta,\varepsilon}$ instead of $RR_{\varepsilon,\delta}\gamma_l$. To simplify notations, we will suppress the dependence on $\varepsilon$ and $\delta$ and write $\gamma_{r}$ instead of $\gamma_{r,\delta,\varepsilon}$ (and similarly for $\tilde{\Gamma}_{r,\delta,\varepsilon}$, $h_{r,\delta}$ and $\theta_{r,\delta}$).  

Hence we can write again
\[ h_r=\cos\delta\left(\frac{-\i+\k}{\sqrt{2}}\right)+\sin\delta\left(\frac{\i+\k}{\sqrt{2}}\right) \]
and the final lifted Frenet frame of $\gamma_r$ is
\[ \tilde{\Gamma}_r(1)=\exp\left((2\pi-\varepsilon) h_r\right)=\exp\left(-\varepsilon h_r\right). \]
Let us first compute to which Bruhat cell the image of $(\tilde{\Gamma}_l(1),\tilde{\Gamma}_r(1))=(\1,\tilde{\Gamma}_r(1))$ under the universal cover projection $\Pi_4 : \SS^3 \times \SS^3  \simeq \mathrm{Spin}_4 \rightarrow \mathrm{SO}_4$ belongs. Using the explicit expression of the map $\Pi_4$ (see Subsection~\ref{s21}), we can compute $\Pi_4(\1,\tilde{\Gamma}_r(1))$ and we find that it is equal to the matrix

\begin{equation*}
\Pi_4(\1,\tilde{\Gamma}_r(1))=
\begin{pmatrix}
P_1 & P_2 & P_3 & P_4
\end{pmatrix},
\end{equation*}
where the columns $P_i$, for $1 \leq i \leq 4$, are given by

\begin{equation*}
P_1=
\begin{pmatrix} 
 \cos(\varepsilon)\\ 
  (-\cos\delta+\sin\delta)\frac{\sin\varepsilon}{\sqrt{2}} \\
 0\\
(\cos\delta+\sin\delta)\frac{\sin\varepsilon}{\sqrt{2}}
\end{pmatrix},
\quad P_2=
\begin{pmatrix} 
(\cos\delta-\sin\delta)\frac{\sin\varepsilon}{\sqrt{2}} \\ 
 \cos(\varepsilon)\\
  (-\cos\delta-\sin\delta)\frac{\sin\varepsilon}{\sqrt{2}}\\
  0
\end{pmatrix},
\end{equation*}

\begin{equation*}
P_3=
\begin{pmatrix} 
0 \\ 
  (\cos\delta+\sin\delta)\frac{\sin\varepsilon}{\sqrt{2}} \\
 \cos(\varepsilon)\\
(\cos\delta-\sin\delta)\frac{\sin\varepsilon}{\sqrt{2}}
\end{pmatrix},
\quad P_4=
\begin{pmatrix} 
(-\cos\delta-\sin\delta)\frac{\sin\varepsilon}{\sqrt{2}} \\ 
0 \\
 (-\cos\delta+\sin\delta)\frac{\sin\varepsilon}{\sqrt{2}} \\
  \cos(\varepsilon) 
\end{pmatrix}.
\end{equation*}

Since $\varepsilon>0$ and $\delta>0$ are small, in particular $0<\varepsilon<\pi$ and $0<\delta<\pi/4$, one can check that this matrix is Bruhat-equivalent to the matrix 
\[ \begin{pmatrix}
0 & 0 & 0 & -1 \\
0 & 0 & 1 & 0 \\
0 & -1 & 0 & 0 \\
1 & 0 & 0 & 0 
\end{pmatrix} \in \mathrm{SO}_4. \]
Therefore $(\1,RR_{\varepsilon,\delta,\gamma_l}(\1))$ is Bruhat-equivalent to $(\1,-\k)$ for the specific curve $\gamma_l$ we choose.

To conclude, observe that for the curve $\gamma_l$ we choose, the final lifted Frenet frame of $(\gamma_l,RR_{\varepsilon,\delta}\gamma_l)$ belongs to an open cell. Using this observation, and the fact that for any curve $\gamma \in \mathcal{L}\SS^2(\1)$, the curve $RR_{\varepsilon,\delta}\gamma$ is obtained from $\gamma$ by relaxing its geodesic curvature essentially in a small $\varepsilon$-neighborhood of $\gamma(0)=\gamma(1)$ (outside this neighborhood the geodesic curvature is only slightly altered), we deduce that for any curve $\gamma \in \mathcal{L}\SS^2(\1)$, the final lifted Frenet frame of $(\gamma,RR_{\varepsilon,\delta}\gamma)$ belongs   to the same open cell as the final lifted Frenet frame of $(\gamma_l,RR_{\varepsilon,\delta}\gamma_l)$. This shows that $(\1,RR_{\varepsilon,\delta,\gamma}(\1))$ is Bruhat-equivalent to $(\1,-\k)$ for any curve $\gamma \in \mathcal{L}\SS^2(\1)$.

\end{proof}

We will use Bruhat cells to remove the dependence on $\varepsilon$, $\delta$ and $\gamma$ from the final lifted Frenet frame $(\pm\1,RR_{\varepsilon,\delta,\gamma} (\pm \1))$.

Recall from Proposition~\ref{bruhathomeo} that there exist natural homeomorphisms

\[ T_{\varepsilon,\delta}^+ : \mathcal{L}\SS^3(\1,RR_{\varepsilon,\delta,\gamma} (\1)) \rightarrow \mathcal{L}\SS^3(\1,-\k), \]
\[ T_{\varepsilon,\delta}^- : \mathcal{L}\SS^3(-\1,RR_{\varepsilon,\delta,\gamma} (-\1)) \rightarrow \mathcal{L}\SS^3(-\1,\k).  \]

Let us now make the following definition.

\begin{definition}\label{relaxcurve2}
For $\gamma \in \mathcal{L}\SS^2(\pm \1)$ and $\varepsilon,\delta>0$ sufficiently small, we define
\[ \hat{\gamma}=T^{\pm}_{\varepsilon,\delta}(\hat{\gamma}_{\varepsilon,\delta}) \in \mathcal{L}\SS^3(\pm \1,\mp \k). \]
\end{definition}

Similarly, for a continuous map $\alpha : K \rightarrow \mathcal{L}\SS^2(\1)$ (respectively a continuous map $\alpha : K \rightarrow \mathcal{L}\SS^2(-\1)$), we define a continuous map  $\hat{\alpha} : K \rightarrow \mathcal{L}\SS^3(\1,-\k)$ (respectively a continuous map  $\hat{\alpha} : K \rightarrow \mathcal{L}\SS^3(-\1,\k)$) by setting $\hat{\alpha}(s)=\widehat{\alpha(s)}$. 

\begin{proposition}\label{propp}
Let $\alpha : K \rightarrow \mathcal{L}\SS^3(\1,-\k)$ (respectively $\alpha : K \rightarrow \mathcal{L}\SS^3(-\1,\k)$) a continuous map. Assume that $\alpha=\hat{\beta}$ for some continuous map $\beta : K \rightarrow \mathcal{L}\SS^2(\1)$ (respectively $\beta : K \rightarrow \mathcal{L}\SS^3(-\1)$). Then $\alpha^\#$ is homotopic in $\mathcal{L}\SS^3(\1,-\k)$ (respectively in $\mathcal{L}\SS^3(-\1,\k)$) to $\widehat{\beta^\ast}$.   
\end{proposition}

\begin{proof}
This follows easily from our definitions of $\alpha^\ast$ (in the case $n=2$), $\alpha^\#$ (in the case $n=3$) and $\hat{\alpha}$.
\end{proof}

\subsection{Proof of Theorem~\ref{th5}}\label{s85}

Now we define $\hat{\mathcal{M}}_k(z_l,z_r)$ to be the set of curves $\gamma=(\gamma_l,\gamma_r) \in \mathcal{L}\SS^3(z_l,z_r)$ such that $\gamma_l \in \mathcal{M}_k(z_l)$. Exactly as in Lemma~$7.1$ of~\cite{Sal13}, we have the following result.

\begin{proposition}\label{manifold2}
The set $\hat{\mathcal{M}}_k(z_l,z_r) \subset \mathcal{L}\SS^3(z_l,z_r)$
is a closed submanifold of  codimension $2k-2$
with trivial normal bundle.
\end{proposition} 

Notice that, while $\mathcal{M}_k$ is contractible,
it is not clear whether $\hat{\mathcal{M}}_k(z_l,z_r)$ also is;
fortunately, we do not need to settle this question.

\begin{proof}
The fact that $\hat{\mathcal{M}}_k(z_l,z_r)$ is a closed set
follows directly from the closure of $\mathcal{M}_k$.

In the proof of Lemma~$7.1$ in \cite{Sal13},
a tubular neighborhood for $\mathcal{M}_k$
is explicitely constructed.
The exact same construction (using $\gamma_l$ only)
still works, and obtains a basis for the normal bundle,
which is therefore trivial.
\end{proof}


As before, we can then associate to $\hat{\mathcal{M}}_k(z_l,z_r)$ a cohomology class $\hat{m}_{2k-2} \in H^{2k-2}(\mathcal{L}\SS^3(z_l,z_r),\R)$ by counting intersection with multiplicity.

In order to prove Theorem~\ref{th5}, we will need the following proposition.

\begin{proposition}\label{homot}
Let $\alpha_0,\alpha_1 : K \rightarrow \mathcal{L}\SS^3(\1,-\k)$ (respectively $\alpha_0,\alpha_1 : K \rightarrow \mathcal{L}\SS^3(-\1,\k)$) be two continuous maps. Assume that $\alpha_0=\hat{\beta_0}$ and $\alpha_1=\hat{\beta_1}$ for some continuous map $\beta_0,\beta_1 : K \rightarrow \mathcal{L}\SS^2(\1)$ (respectively $\beta_0,\beta_1 : K \rightarrow \mathcal{L}\SS^2(-\1)$). Then $\alpha_0$ and $\alpha_1$ are homotopic in $\mathcal{L}\SS^3(\1,-\k)$ (respectively in $\mathcal{L}\SS^3(-\1,\k)$) if and only if $\beta_0$ and $\beta_1$ are homotopic in $\mathcal{L}\SS^2(\1)$ (respectively in $\mathcal{L}\SS^2(-\1)$).    
\end{proposition} 

\begin{proof}
It is sufficient to consider the case where $\alpha_0,\alpha_1 : K \rightarrow \mathcal{L}\SS^3(\1,-\k)$ (the case where $\alpha_0,\alpha_1 : K \rightarrow \mathcal{L}\SS^3(-\1,\k)$ is, of course, the same). We know that $\alpha_0=\hat{\beta_0}$ and $\alpha_1=\hat{\beta_1}$ for some continuous map $\beta_0,\beta_1 : K \rightarrow \mathcal{L}\SS^2(\1)$.

On the one hand, if $H$ is a homotopy between $\alpha_0$ and $\alpha_1$, it can be decomposed as $H=(H_l,H_r)$, and it is clear that $H_l$ gives a homotopy between $\beta_0$ and $\beta_1$. On the other hand, if $H$ is a homotopy between $\beta_0$ and $\beta_1$, $\hat{H}$ provides a homotopy between $\alpha_0$ and $\alpha_1$. 
\end{proof}

Theorem~\ref{th5} will now be an easy consequence of the following proposition. Here the functions $h_{2k-2} : \SS^{2k-2} \rightarrow \mathcal{L}\SS^2((-\1)^k)$ are as in~\cite{Sal13} and in Subsection~\ref{s83} above.

\begin{proposition}\label{nico2}
Consider an integer $k \geq 2$. If $k$ is even, the maps
\[ \hat{h}_{2k-2} : \SS^{2k-2} \rightarrow \mathcal{L}\SS^3(\1,-\k) \]
are homotopic to constant maps in $\mathcal{G}\SS^3(\1,-\k)$, but satisfy $\hat{m}_{2k-2}(\hat{h}_{2k-2})=\pm 1$.
If $k$ is odd, the maps
\[ \hat{h}_{2k-2} : \SS^{2k-2} \rightarrow \mathcal{L}\SS^3(-\1,\k) \]
are homotopic to constant maps in $\mathcal{G}\SS^3(\1,-\k)$, but satisfy $\hat{m}_{2k-2}(\hat{h}_{2k-2})=\pm 1$.
\end{proposition} 

\begin{proof}
Let us consider the case where $k$ is even (the case where $k$ is odd is exactly the same). By definition of $\hat{h}_{2k-2}$ and $\hat{m}_{2k-2}$, it is clear that
\[ \hat{m}_{2k-2}(\hat{h}_{2k-2})=m_{2k-2}(h_{2k-2}) \]
and therefore we have $\hat{m}_{2k-2}(\hat{h}_{2k-2})=\pm 1$.

It remains to prove that $\hat{h}_{2k-2}$ is homotopic to a constant map in $\mathcal{G}\SS^3(\1,-\k)$. From  Lemma~$7.3$ on~\cite{Sal13}, we know that $h_{2k-2}$ is homotopic to a constant map in $\mathcal{G}\SS^2(\1)$. Therefore this implies that $h_{2k-2}^\ast$ is homotopic to a constant map in $\mathcal{L}\SS^2(\1)$ (see Proposition $6.4$ on~\cite{Sal13}). Let us denote by $c : K \rightarrow \mathcal{L}\SS^2(\1)$ this constant map; then         $\hat{c} : K \rightarrow \mathcal{L}\SS^3(\1,-\k)$ is also a constant map. Now, by Proposition~\ref{propp}, $\hat{h}_{2k-2}^\#$ is homotopic in $\mathcal{L}\SS^3(\1,-\k)$ to $\widehat{h_{2k-2}^\ast}$. Since $h_{2k-2}^\ast$ is homotopic to $c$ in $\mathcal{L}\SS^2(\1)$, it follows from Proposition~\ref{homot} that $\widehat{h_{2k-2}^\ast}$ is homotopic to the constant map $\hat{c}$ in $\mathcal{L}\SS^3(\1,-\k)$, and so $\hat{h}_{2k-2}^\#$ is homotopic to the constant map $\hat{c}$ in $\mathcal{L}\SS^3(\1,-\k)$. Using Proposition~\ref{Genericas2} (we will only use the easy direction which follows from Proposition~\ref{Genericas2}), this shows that $\hat{h}_{2k-2}$ is homotopic to a constant map in $\mathcal{G}\SS^3(\1,-\k)$, which is what we wanted to prove.    
\end{proof}

\medskip

Therefore, given a integer $k \geq 2$, $\hat{h}_{2k-2}: \SS^{2k-2} \rightarrow \mathcal{L}\SS^3((-\1)^k,(-\1)^{(k-1)}\k)$ defines extra generators in $\pi_{2k-2}(\mathcal{L}\SS^3((-\1)^k,(-\1)^{(k-1)}\k))$ as compared to \\ $\pi_{2k-2}(\mathcal{G}\SS^3((-\1)^k,(-\1)^{(k-1)}\k))$.

Using Proposition~\ref{nico2}, it will be easy to conclude.

\begin{proof}[Proof of Theorem~\ref{th5}]
First let us recall that the inclusion $ \mathcal{L}\SS^3(z_l,z_r) \subset \mathcal{G}\SS^3(z_l,z_r) $ always induces surjective homomorphisms between homology groups with real coefficients (\cite{SS12}). Also, for any $j \geq 1$, we have injective homomorphisms between cohomology groups with real coefficients 
\[ H^j(\mathcal{G}\SS^3(z_l,z_r),\R) \simeq H^j(\Omega(\SS^3 \times \SS^3), \R) \rightarrow H^j(\mathcal{L}\SS^3(z_l,z_r),\R).  \]
In our case, this implies
\begin{equation*}
\mathrm{dim}\; H^j(\mathcal{L}\SS^3(-\mathbf{1},\1),\R)=\mathrm{dim}\; H^j(\mathcal{L}\SS^3(\mathbf{1},-\k),\R) \geq 
\begin{cases}
0 & j\;\mathrm{odd} \\
l+1 & j=2l, 
\end{cases}
\end{equation*}
and
\begin{equation*}
\mathrm{dim}\; H^j(\mathcal{L}\SS^3(\mathbf{1},-\1),\R)=\mathrm{dim}\; H^j(\mathcal{L}\SS^3(-\mathbf{1},\k),\R) \geq
\begin{cases}
0 & j\;\mathrm{odd} \\
l+1 & j=2l.
\end{cases}
\end{equation*}
But now Proposition~\ref{nico2} gives, for $k \geq 2$ even, an extra element $\hat{m}_{2k-2}$ in the cohomology of degree $2k-2$ for $\mathcal{L}\SS^3(-\mathbf{1},\1) \simeq \mathcal{L}\SS^3(\mathbf{1},-\k)$. Writing $j=2l$, this gives an extra element when $j=2l$ with $l$ odd, therefore 
\begin{equation*}
\mathrm{dim}\; H^j(\mathcal{L}\SS^3(-\mathbf{1},\1),\R) \geq 
\begin{cases}
0 & j\;\mathrm{odd} \\
l+2 & j=2l, \; l\;\mathrm{odd}  \\
l+1 & j=2l, \; l\;\mathrm{even}.
\end{cases}
\end{equation*}
Similarly, Proposition~\ref{nico2} gives, for $k \geq 2$ odd, an extra element $\hat{m}_{2k-2}$ in the cohomology of degree $2k-2$ for $\mathcal{L}\SS^3(\mathbf{1},-\1) \simeq \mathcal{L}\SS^3(-\mathbf{1},\k)$. Writing $j=2l$, this gives an extra element when $j=2l$ with $l$ even, and so
\begin{equation*}
\mathrm{dim}\; H^j(\mathcal{L}\SS^3(\mathbf{1},-\1),\R) \geq
\begin{cases}
0 & j\;\mathrm{odd} \\
l+1 & j=2l, \; l\;\mathrm{odd}  \\
l+2 & j=2l, \; l\;\mathrm{even}.
\end{cases}
\end{equation*}
This ends the proof.
\end{proof}


\bigskip

\parindent=0pt
\parskip=0pt
\obeylines

Emília Alves and Nicolau C. Saldanha
emilia@mat.puc-rio, saldanha@puc-rio.br
Departamento de Matemática, PUC-Rio
Rua Marquês de São Vicente 225, Rio de Janeiro, RJ 22451-900, Brazil

\end{document}